\documentclass[12pt]{amsart}

\usepackage{tikz-cd} 

\usepackage{latexsym}
\usepackage{amssymb}
\usepackage[active]{srcltx}

\usepackage{color}

\usepackage{epsfig}

\newcommand{\I}{\mathrm{i}}

\newcommand{\dd}{\mathrm{d}}
\newcommand{\ee}{\mathrm{e}}

\newcommand{\R}{{\mathbb R}}
\newcommand{\C}{{\mathbb C}}

\newcommand{\Z}{{\mathbb Z}}

\newcommand{\Tr}{{\operatorname{Tr}\,}}

\newcommand{\half}{{\frac{1}{2}}}

\newcommand{\supp}{{\operatorname{Supp\,}}}
\renewcommand{\phi}{\varphi}

\newcommand{\acal}{\mathcal{A}}

\newcommand{\ccal}{\mathcal{C}}

\newcommand{\hcal}{\mathcal{H}}

\newcommand{\lcal}{\mathcal{L}}
\newcommand{\pcal}{\mathcal{P}}

\newcommand{\ncal}{\mathcal{N}}

\newcommand{\rcal}{\mathcal{R}}

\newcommand{\corri}[1]{#1}

\newcommand{\Poincare}{Poincar{\'e} }
\newcommand{\FE}{\mathrm{FE}}
\newcommand{\rmi}{\mathrm{i}}

\newtheorem{theo}{Theorem}[section]
 \newtheorem{lem}[theo]{Lemma}
 \newtheorem{cor}[theo]{Corollary}
 \newtheorem{prop}[theo]{Proposition}

 \theoremstyle{definition}
 \newtheorem{defin}[theo]{Definition}
 \newtheorem{example}[theo]{Example}
 \newtheorem{rem}[theo]{Remark}

\title{A Gutzwiller trace  formula for stationary space-times}

\author{Alexander Strohmaier}
\author{Steve Zelditch}

\address{School of Mathematics \\
University of Leeds\\
Leeds, LS2 9JT, UK}
\email{a.strohmaier@leeds.ac.uk}

\address{Department of Mathematics, Northwestern  University,
Evanston, IL 60208-2370, USA} \email{
zelditch@math.northwestern.edu}

\thanks{Research partially supported by  NSF grant DMS-1810747.}

\begin{document}
\maketitle

\begin{abstract}
 We give a relativistic generalization of  the Gutzwiller-Duistermaat-Guillemin trace formula for the wave group of  a compact Riemannian manifold  to globally hyperbolic  stationary space-times with compact Cauchy hypersurfaces. We introduce several (essentially equivalent) notions of trace of self-adjoint operators  
 on the null-space $\ker \Box$ of the wave operator  and define $U(t)$ to be translation by the flow $e^{t Z}$ of the timelike Killing vector field $Z$ on $\Box$. The spectrum of $Z$ on $\ker \Box$ is discrete and the singularities
 of $\rm{Tr} \ee^{t Z} |_{\ker \Box}$ occur at periods of periodic orbits of $\exp t Z$ on the symplectic manifold of null geodesics. The trace formula gives
 a Weyl law for the eigenvalues of $Z$ on $\ker \Box$.
\end{abstract}

{\small
\tableofcontents}

\section{Introduction}

For a closed connected Riemannian manifold $(\Sigma,h)$ the
distribution traces of the wave groups
\begin{equation} \label{U(t)} U(t) := \begin{pmatrix} \cos t \sqrt{-\Delta} & \frac{\sin t \sqrt{\Delta}}{\sqrt{-\Delta}} \\ & \\
-\sqrt{-\Delta} \sin t \sqrt{-\Delta} & \cos t \sqrt{-\Delta}
\end{pmatrix},\;\; {\rm and}\;\; V(t) : = \exp (\rmi t \sqrt{-\Delta})
\end{equation} are
$$\mathrm{Tr} \; U (t) = 2 \sum_{ j}  \cos \lambda_j t
= 2 \Re \sum_j \ee^{\rmi t \lambda_j}, \;\; {\mathrm{resp.}}\;\;  \mathrm{Tr} \; V (t) = \sum_j \ee^{\rmi t \lambda_j},
$$
where $(\lambda_j)_{j \in \mathbb{N}_0}$ is the non-decreasing sequence of eigenvalues of $\sqrt{-\Delta}$ repeated according to their multiplicities.
It was proved by Chazarain \cite{Ch74} and by Duistermaat-Guillemin \cite{DG75} that  the  positive singular points of $\Tr U (t)$ occur when $t$ lies in the length spectrum $\mathrm{Lsp}(\Sigma,h)$,
i.e. the set of lengths $L_{\gamma}$ of closed geodesics $\gamma$. The analysis of the singularity at $t=0$ was carried out by H\"ormander in \cite{Ho2} and famously gives rise to the Weyl law with its sharp remainder estimate.
As reviewed in Section \ref{DG-Product}, when the closed geodesics are non-degenerate, 
$\Tr V(t)$ admits a singularity expansion around $0<t = L_{\gamma}$ of the form  
\begin{equation} \label{TrUt} \Tr V (t) \sim  a_{\gamma, -1}  (t - L_{\gamma} + \rmi 0)^{-1}, \;\; \rm{with}\;
  a_{\gamma, -1} = \corri{\frac{1}{2 \pi \rmi}}
\frac{\ee^{-\frac {\rmi \pi }{2} m_\gamma }L_\gamma^{\#}}{|\det(\mathrm{id}-P_{\gamma})|^{\half}}, \end{equation} where  $P_\gamma$ is the  linear
Poincar\'e map, $L_{\gamma}$ is the length of $\gamma$, $L_{\gamma}^{\#}$ is the primitive period and $m_{\gamma}$ is the Maslov index (discussed below). There is a dual semi-classical expansion for the Fourier transform of $\Tr V(t)$ which preceded the rigorous mathematical work, and this expansion is often called the Gutzwiller trace formula \cite{Gutz71, M92}
or the Poisson relation.\footnote{ We use the first term to avoid confusion with Poisson integral formulae.}

As described above, the formula  is manifestly non-relativisitic.
The propagator $V(t)$ is a solution operator  for the homogeneous wave equation $\Box
u = 0,$  and both  $V(t)$ and its generator $\sqrt{-\Delta}$ are non-relativistic;  their generalizations to Schr\"odinger equations in \cite{Gutz71} is a branch of non-relativistic quantum mechanics. Second, the geometry of the terms of the singularity trace formula is the Riemannian geometry of $(\Sigma,h)$ and the symplectic geometry of the geodesic flow $G^t$ on the unit cosphere bundle $S^* \Sigma$.  
The purpose of this article is to prove a relativistic Gutzwiller trace formula
for globally hyperbolic stationary spacetimes $(M, g)$, i.e. globally hyperbolic spacetimes with a complete timelike Killing vector field $Z$. The main result, Theorem \ref{GUTZTH}, is a singularities trace formula for the trace of  the  translation operator by the flow $ \ee^{t Z}$ of the Killing vector field $Z$ acting on the nullspace $\ker \Box$ of a wave operator $\Box$. On the classical level, the Killing flow acts on the contact manifold $\ncal_p$ of unparametrized null-geodesics of $(M, g)$, and the singular times are the periods of periodic orbits of $e^{t Z}$ on $\ncal_p$.   As a corollary, one gets a Weyl counting formula (Corollary \ref{WEYLCOR}) with error estimate for the number of eigenvalues $\leq T$ of the Killing 
vector field on $\ker \Box$. One can also consider conformal timelike Killing vector fields acting on the null space of the conformal d'Alembert operator. However, this case can easily be reduced to the one we consider by a conformal change of the metric.

 \subsection{Statement of results}


We will assume that $(M,g)$ is a spatially compact globally hyperbolic stationary space-time of dimension $n$.
Let $\Box_g$ denote the d'Alembert operator that is given in local coordinates as
$$
 \Box_g = -\frac{1}{\sqrt{|g|}} \partial_i \left(\sqrt{|g|} g^{ik}\partial_k \right),
$$
where we have used Einstein's sum convention.
Let $Z$ be the associated timelike Killing vector field. We can think of $Z$
as a first order differential operator that coincides with its Lie derivative $\lcal_Z$ on functions. Put $D_Z =
\frac{1}{i} \lcal_Z$.
More generally we consider a potential $V \in C^\infty(M)$ with $D_Z V=0$ and the operator
\begin{gather}
 \Box = \Box_g + V.
\end{gather}
This includes interesting examples of the form $\Box = \Box_g + m^2 + \kappa R$ where $\kappa,m \in \R$ and $R$ denotes the scalar curvature. Then, since we assumed that $D_Z V=0$, the operator $D_Z$ commutes with $\Box$:
\begin{equation} \label{COMMUTE} [D_Z, \Box] = 0.\end{equation}
Denote by $\Psi^t  = \exp t Z$ the Killing flow generated by $Z$. This flow acts on functions $u$ by pull-back $\Psi^t u = (\Psi^t)^* u = \ee^{\rmi t D_Z} u$.
Then $$D_Z: \ker \Box \to \ker \Box, \;\; {\rm and}\;\; \Psi_t:   \ker \Box \to \ker \Box $$  where $\ker \Box$ is the solution space of $\Box u =
0$  on $M.$ 

The eigenfunctions of $D_Z$ in $\ker \Box$ are joint eigenfunctions,
$$\left\{ \begin{array}{l}  \Box u = 0, \\ \\  D_Z u = \lambda u. \end{array} \right. $$
We will show that the spectrum $\mathrm{Sp}(\Box,D_Z)$ consists of at most finitely many non-real eigenvalues and an infinite discrete set of real eigenvalues (Theorem \ref{ponteigs}). The spectrum is also symmetric with respect to complex conjugation
and reflection at the origin.
By elliptic regularity the eigenfunctions are smooth even when considered on the space of distributional solutions. 
We can therefore, without any loss of generality, consider the eigenvalue problem as posed on the space of smooth solutions of the wave equation.

The first issue in defining a relativistic trace formula is to define a suitable
notion of trace of operators on $\ker \Box$. This would seem to require  an inner product on $\ker \Box$. 
In fact, it is sufficient to  define a Hilbert space topology on $\ker \Box$ since the trace is the same for equivalent Hilbert inner products. 
Suitable Hilbert space topologies induced from finite energy spaces are
discussed in Section \ref{kerBoxSect}, and following this we endow $\ker \Box$ with the topology of a Hilbert space.
The resulting trace, 
 \begin{equation} \label{TRD} \Tr U(t) = \Tr  \ee^{\rmi t D_Z} |_{\ker \Box} \end{equation} 
 is defined in Section \ref{WTSECT}  (see \eqref{TRUtDEF}).  Note that in general it is not possible to define an inner product that is both positive definite and invariant under the Killing flow. Using the stress energy tensor one can however define an invariant non-definite sesquilinear form on $\ker \Box$ (see Section \ref{ESECT}). This energy form is natural and invariant under the Killing flow. The completion (after possibly dividing out a null space) leads to a Pontryagin space. One can then use spectral theory on Pontryagin spaces to analyse the generator of the Killing flow. The appearance of Krein and Pontryagin spaces in the analysis of the Klein-Gordon operator has been noticed before in slightly different contexts (\cite{LNT06,LNT08,GGH17}).
    
Other types of inner products on $\ker \Box$ have been studied in many papers on quantum field theory on a curved space time.  $\ker \Box$ is naturally a symplectic vector space, so a compatible inner product is defined once a (linear) complex structure $J$ is defined.  A complex structure is equivalent to splitting $\ker \Box \otimes \C$ into $\pm \rmi$ eigenspaces of $J$. In quantum field theory this is for example achieved by a frequency splitting procedure giving rise to a vacuum-state.
 There is a natural definition in the stationary case using the positive/negative eigenvalues of $D_Z$. Such frequency splitting procedures are known to give rise to Hadamard states. In Section
  \ref{Hadamard} we explain the relation between Hadamard states and complex structures. Using the mode expansion we show that any invariant complex structure leads to a Hadamard state. In particular this means that any invariant pure quasifree state is a Hadamard state.

The next issue is to find analogues for the geodesic flow, periodic orbits and the symplectic geometry of the Poincar\'e map. 
 From the viewpoint of geometric
  quantization, the Hilbert space $\ker \Box$ is a quantization
  of the space $\ncal$ of the null-bicharacteristics of $\Box$. The Hamiltonian of the null-bicharacteristic
  flow is the energy function $$\frac{1}{2}\sigma_{\Box}(x, \xi) = \frac{1}{2} |\xi|_g^2, $$
  where $|\xi|_g^2$ is the Lorentzian `norm' squared. We  use the factor $\frac{1}{2}$ here to make sure that this flow is identical to the geodesic flow $G^t$ on the cotangent bundle.
  To be more precise, let $\mbox{Char}(\Box) = \{(x, \xi) \in T^* M \setminus 0: \sigma_{\Box}(x, \xi) = 0 \}$.  We denote the restriction of $G^t$ 
  to ${\rm Char}(\Box)$ by $G^t_0$.  $\mbox{Char}(\Box)$  is a (co-isotropic)  hypersurface whose null-foliation consists of orbits of $G^t_0$, i.e. of scaled null-geodesics (see Sec. \ref{sec1.1} for a definition).
The space $\ncal$, of scaled null-geodesics, is naturally a non-compact symplectic manifold.

The flow $\Psi^t$ of the Killing vector field commutes
with the null-bicharacteristic flow and there defines  a quotient
(reduced) symplectic  flow on $\ncal$. We denote the quotient flow  by
$\Psi_{\ncal}^t$.  
\begin{lem}\label{K} $\Psi^t_{\ncal}$ is a Hamiltonian flow with Hamiltonian 
$$H(\zeta) = \xi (Z), \;\; {\rm where}\; \zeta = \{G^t(x, \xi), t \in \R\}. $$
The value $\xi(Z)$ is independent of the lift of $\zeta$ to $(x, \xi)$.
\end{lem} 
The set of scaled null-geodesics is a symplectic cone and the quotient by its $\R_+$-action is the space $\ncal_p$ of unparametrized null-geodesics.
In our case, since $Z$ is timelike, the Hamiltonian is positive and homogeneous. Hence, for any $E>0$ the contact manifold $\ncal_p$ can be identified with the 
energy surface $\ncal_{E}: = \{\zeta: H(\zeta) = E\}$. As with any Hamiltonian flow, $\Psi^t_{\ncal}$ preserves level sets of $H$ and therefore also acts on 
$\ncal_{E}: = \{\zeta: H(\zeta) = E\}$. The induced flow on the quotient space $\ncal_p$ will be denoted by $\Psi^t_{\ncal_p}$. The identification $\ncal_p$ with $\ncal_E $ is equivariant.

We then  define the periods  and  periodic points of $\Psi^t_{\ncal_p}$ by
\begin{equation} \label{PERIODS} \pcal: = \{T \not= 0: \exists \zeta \in \ncal_p: \Psi^T_{\ncal_p}(\zeta) = \zeta\}, \;\; \pcal_T =  \{ \zeta \in \ncal_p: \Psi^T_{\ncal_p}(\zeta) = \zeta\}.  \end{equation}
Each periodic point $\zeta$ of period $T$ has an associated orbit $\cup_{t} \Psi^t_{\ncal_p}(\zeta)$. We will call such an orbit together with the period $T$ a periodic orbit 
$\gamma$ with period $T$. Thus, a periodic orbit $\gamma$ has an associated period $T_\gamma$ and an associated primitive period $T^\#_\gamma$. 
If $\zeta$ is in the orbit $\gamma$ the primitive period is by definition the smallest non-negative $t \in \R$ such that $\Psi^t_{\ncal_p}(\zeta) = \zeta$.
The period $T$ is an integer multiple of the primitive period $T_{\gamma}^{\#}$.

 \begin{prop}\label{POISSON}  Let $(M, g)$ be a spatially compact globally hyperbolic stationary spacetime. Then $ \Tr  \ee^{\rmi t D_Z} |_{\ker \Box} $ is a distribution on $\R$,
  and its singular support is a subset of $\{0\} \cup \pcal$. \end{prop}
Given $\zeta \in \pcal_T$, there is a local 
symplectic transversal $S_{\zeta} \subset \ncal_E$ to $\Psi^t_{\ncal}$ and
a local first return map 
$\Phi_{\zeta}(y) = \Psi^{T_{\zeta}(y)}(y), $ where $y \in S_{\zeta}$ and  $\Psi_{\ncal}^t$-orbit of $y$ to
$S_{\zeta}$. The first return time $T_{\zeta}(y)$ is well-defined in a sufficiently small neighborhood of $\zeta \in S_{\zeta}$ and chosen such that $T=T_{\zeta}(\zeta)$.
 The 
 linear Poincar\'e map is then defined by
\begin{equation} \pcal_{\zeta}: T_{\zeta} S_{\zeta} \to T_{\zeta} S_{\zeta}, \;\; \pcal_{\zeta} = D_{\zeta} \Phi_{\zeta}. \end{equation}
The conjugacy class of the \Poincare map is independent of $E$ and of the point $\zeta$. Hence for a periodic orbit $\gamma$ it makes sense to define
$\det(I - P_{\gamma})$ as $\det(I - P_{\zeta})$, where $\zeta$ is a point in the orbit $\gamma$. 
The orbit is called {\it non-degenerate} if $\det(I-P_{\gamma}) \not=0. $
Orbits are classified as non-degenerate elliptic, hyperbolic and so on as for any Hamiltonian flow (see for instance \cite{Kl, FrG91}).
  
 We can now state the two main results:

\begin{theo} \label{HOERM}
For general spatially compact stationary globally hyperbolic spacetimes we have that
$$\Tr \ee^{\rmi t
 D_Z} \; |_{Char(\Box)}
 = e_0(t) + \psi(t)$$ where
$\psi$ is a distribution that is smooth near $0$, and $e_0(t)$ is a Lagrangian distribution with
singulariy at $t = 0$ of the form
$$e_0(t) \sim 2 (2 \pi)^{-n+1} (n-1) \mathrm{Vol}(\ncal_{H\leq1}) \mu_{n-1}(t) + c_1  \mu_{n-2}(t) 
+ \ldots,$$
where the homogeneous distribution $\mu_k(t)$ is defined by the oscillatory integral
$$
 \mu_k(t) = \frac{1}{2}\int_{-\infty}^\infty \mathrm{e}^{-\rmi t \tau} |\tau|^{k-1} \mathrm{d} \tau,
$$
(see for example \cite[Vol I]{HoI-IV} for properties of these distributions).
\end{theo}

\begin{theo} \label{GUTZTH}
Let $T \in \pcal$ and assume that the fixed point set $\pcal_T$ of $\Psi_{\ncal_p}^s$ on
$\ncal_p$ is non-degenerate, i.e.  it is a finite union of non-degenerate periodic
orbits $\gamma$. Then,
$$\Tr \ee^{\rmi t
 D_Z} \; |_{Char(\Box)}
 = 2 \sum_{\gamma: T_\gamma = T} \Re (e_{\gamma}(t)) + \psi_T, $$ where $\psi_T$ is a distribution smooth near $t=T$ and
$e_{\gamma}(t)$ are Lagrangian distributions with
singularities at $t = T_{\gamma}$. If $\gamma$ is non-degenerate, we have
$$e_{\gamma}(t) \sim \corri{\frac{1}{2 \pi \rmi}} \frac{\ee^{-\rmi \frac{\pi}{2} m_{\gamma}} T^{\#}_\gamma
}{|\det(I - P_{\gamma})|^{\half}} (t - T_{\gamma}+\rmi 0)^{-1} + \ldots,
$$
where $m_{\gamma}$ is the Conley-Zehnder index of the periodic orbit
$\gamma$. The sum is over all periodic orbits of period $T$.
The expansion above is a singularity expansion around $t=T_{\gamma}$.
\end{theo}
The factor $e^{-\rmi \frac{\pi}{2} m_{\gamma}} $ is often called the Maslov
factor and $m_{\gamma}$ the Maslov index. In \cite{R91}, the index is clarified and in \cite{M94} it is generalized to all symplectic manifolds and
identified as the Conley-Zehnder index. There is a standard generalization to `clean fixed point sets'  but we will not consider such generalizations here.

As a corollary, using a standard Fourier Tauberian argument, \cite[Lemma 17.5.6]{HoI-IV}, we may derive a Weyl law for the growth of the
spectrum of $D_Z$ in $\mbox{Char}(\Box)$.

\begin{cor}\label{WEYLCOR} 
For general spatially compact stationary globally hyperbolic spacetimes the
spectrum of $D_Z$ in $\mbox{Char}(\Box)$ is
discrete. Moreover the Weyl eigenvalue counting function
$$N_{Z}(\lambda) : = \# \{j: 0\leq \lambda_j \leq \lambda\}, $$
has the asymptotics,
$$N_Z(\lambda) = \frac{1}{(2 \pi)^{n-1}}\; \mathrm{Vol}(\ncal_{H \leq 1}) \lambda^{n-1} + O( \lambda^{n-2}),
$$
as $\lambda \to \infty.$
\end{cor}

As in the product case, we conjecture that the remainder term is $o(\lambda^{n-2})$ when the set of periodic orbits has Liouville measure zero.

\subsection{Outline of the proof}

The proof of Theorems \ref{HOERM} and \ref{GUTZTH} is based on the symbol calculus of
Fourier integral operators and does not require  hard analysis. In particular, it does not use separation of variables.  For this reason, the approach should allow for further results  in relativistic
spectral asymptotics. 

The  well-known important features of globally hyperbolic stationary 
spacetimes used in the proof are:\bigskip

\begin{enumerate}

\item The Cauchy problem for $\Box u = 0$ with Cauchy data on a spacelike Cauchy hypersurface $\Sigma$ is well-posed. There exist global solutions for finite energy Cauchy data.  \bigskip

\item There exist advanced, resp. retarded, Green's functions $E_{\mathrm{ret/adv}}$
satisfying $\Box E_{\mathrm{ret/adv}} = \mathrm{id}$.
The  Green's functions can be used
to give Fourier integral  representation formulae for solutions of $\Box u = 0$ as integrals
over $\Sigma$ in the sense of \cite{R1,R2, F} (see \cite{BGP} for a modern presentation). This allows for a Fourier integral operator representation of the Killing flow translation on null solutions. The Green's functions and representation formulae also induce traces on the null-space of $\Box$.
\bigskip

\item Given one Cauchy hypersurface $\Sigma$, the Killing flow induces
a foliation of $M$ by Cauchy hypersurfaces $\Sigma_t = \ee^{t Z} \Sigma.$
Thus, the unitary evolution operator can be viewed as arising from the 
cobordism of $\Sigma = \Sigma_0$ and $\Sigma_t$, analogous to the elliptic case studied in \cite{BW} and elsewhere. One then uses the Killing flow $e^{t Z}$ to pull back
the Cauchy data on $\Sigma_t$  to $\Sigma_0$.
That is, we   consider the evolution operator as 
 $$CD_0(\Sigma) \to \ker \Box \to  CD(\Sigma_t) \to CD_0(\Sigma)$$
 where the last arrow is translation by $\ee^{t Z}$.  One may visualize the
 operator in terms of the following diagram:
 \begin{equation}
 \begin{tikzcd}\label{DIAGRAMA}
   & \ker \Box \arrow[ld,"CD_0"'] \arrow[rd,"CD_t"] & \\
  \hcal_{(1)}(\Sigma_0) \oplus L^2(\Sigma_0)  &  &  \arrow[ll,"\Psi(-t)"]  \hcal_{(1)}(\Sigma_t) \oplus L^2(\Sigma_t) 
 \end{tikzcd}
 \end{equation}
Here, $CD_t$ denotes the restriction or Cauchy data operator with respect to $\Sigma_t$.
On the symbolic level,  the diagram becomes
\begin{equation}
 \begin{tikzcd}\label{DIAGRAMB}
   & \rm{Char}(\Box) \arrow[ld,"\pi_1"'] \arrow[rd,"\pi_2"] & \\
  \ncal \simeq T^* \Sigma_0  &  &  \arrow[ll,"\ee^{-t Z}"]  \ncal \simeq T^* \Sigma_t
 \end{tikzcd}
 \end{equation}
 where the identifications $\ncal \simeq T^* \Sigma_t$ are to take the 
 tangent co-vector of the null geodesic at $x \in \Sigma_t$ and restrict
 it to $T \Sigma_t$. In this sense, our problem is reminiscent of recent work on restriction theorems for eigenfunctions. 

As in \cite{DG75}, the trace is represented as a composition  $\pi_* \Delta^* \ee^{t Z} \circ \Psi(-t)$ of Fourier integral operators as in the top diagram. Here, as in \cite{DG75},
$\Delta^*$ represents a pull-back to a certain diagonal and $\pi_*$ denotes integration over certain fibers.  
Once we make these
notions precise, the relativistic Gutzwiller-Duistermaat-Guillemin formula follows by the composition calculus of canonical relations and symbols. The main purpose of setting up the problem in this invariant form is that the role of periodic orbits of
$e^{t Z}$ on $\ncal$ comes out naturally.

\end{enumerate}

\subsection{Sign conventions and notations}

Since several statements in this paper depend critically on the correct and consistent sign conventions we explain here our choice and its relation to the literature.
The Fourier transform $\hat f$ of $f \in L^1(\R^n)$ will be defined by $$\hat f(\xi) = \int f(x) \ee^{-\rmi \, x \cdot \xi} \dd x$$ where $x \cdot \xi$ is the Euclidean inner product on $\R^n$.
For the metric we choose the sign convention $(-,+,\ldots,+)$ and the d'Alembert operator $\Box$ has principal part in local coordinates $-\sum_{i,k = 1}^n g^{ik} \partial_i \partial_k$ and therefore its principal symbol is
$\sigma_\Box(\xi)=g(\xi,\xi)$. 
Finally, we work with the field of complex numbers unless otherwise stated: For example $C^\infty(M)$ denotes the space of complex valued smooth functions on $M$. The set of real-valued smooth functions on $M$ is denoted by $C^\infty(M,\R)$. For the sake of computations of principal symbols of operators we identify functions on manifolds with half-densities, using the metric and then view the operators as acting between half-densities. 

\subsection{Comments}

In Section \ref{DG-Product}, we explain how to reformulate
the Gutzwiller-Duistermaat-Guillemin trace formula on the product Lorentz manifold $\R \times \Sigma$ in these terms. 
 Apart from the fact that our results are intrinsically relativistic they are also more general, and cannot be reduced in a straightforward way to the classical trace formula for compact manifolds.
 Instead of a classical eigenvalues equation of the form
 $$
  (\Delta -\lambda^2) \psi =0
 $$
 the stationary eigenvalue problem is equivalent (via separation of variables) to an eigenvalue problem of a quadratic operator pencil of the form
 $$
   ( \Delta - 2 (\I X) \lambda  - \lambda^2) \psi =0,
 $$
where $\Delta$ is a Laplace-type operator on a compact Riemannian manifold, and $X$ a vector field on that manifold. The vector field $X$ is related to the shift vector field and it vanishes in case the spacetime carries a metric of product-type.
We will show that this quadratic pencil factorises into two linear scalar pseudo-differential eigenvalue problems only in case $X$ is a Killing vector field.
Generally this is not the case for stationary space-times.

The appearance of quadratic operator pencils in stationary problems in general relativity has been observed in many situations,
for example for Kerr spacetime (\cite{B00}) or the BTZ black hole in dimension 3 (\cite{BDFK,B00}), or in the context of a general Klein-Gordon equation (\cite{GGH17}). In \cite{F08} the author claims that operator pencils of this form should have a spectrum satisfying Weyl's law, and the example of a rotating Einstein universe is considered. We are not aware of any previous work on singularity expansions of the trace in this context, nor of any rigorous work on a Weyl law and its error.

This paper deals with globally hyperbolic spacetimes with compact Cauchy surface, and
the manifolds we consider do not generally satisfy the vacuum Einstein equations. In fact,
in dimension four all stationary solutions of the vacuum Einstein equations are either
flat or are spatially non-compact ([A00, CM16]).  Our setting is analogous to that of Duistermaat-Guillemin \cite{DG75}, which considers the trace of the wave group on  all compact Riemannian manifolds, not just on Einstein manifolds. Our results similarly encompass all Lorentzian stationary globally hyperbolic spatially compact spacetimes.

\subsection{Related problems} The relativistic approach of this article
applies to many other spectral problems related to Weyl's law and the 
wave trace on a compact Riemannian manifold. Some potential applications
could include a relativistic analogue of quantum ergodicity for globally hyperbolic spacetimes for which $e^{t Z}$ acts ergodically on the space
$\ncal_p$ of unparametrized null-geodesics. A slightly different problem is that of semi-classical mass asymptotics in which both the energy and the mass increase simultaneously. A  Weyl formula in this context that was inspired by this paper has meanwhile appeared in \cite{SZ20}.

\subsection{Generalization to spatially non-compact stationary spacetimes}

It is natural to try to generalize the trace formula of this article to globally
hyperbolic stationary spacetimes with non-compact Cauchy hypersurfaces. 
The spectrum of $D_Z |_{\ker \Box}$ then becomes continuous and the formulae of this article do not apply. 
In the ultra-static (product) case, this generalization amounts to trace formulae,  scattering theory and resonances for short range (e.g. compact) metric or potential  perturbations of the Euclidean Laplacian on $\R^n$. Scattering for the exterior of several compact obstacles also falls into this class of problems.  Trace formulae in scattering theory, involving either the scattering phase or resonances, were originally  given in works of Krein, Jensen-Kato, Majda-Ralston, Melrose and many others;  singularities trace formulae in this setting were originally studied by Bardos-Guillot-Ralston. Their  results  require a well-developed scattering theory and involve many technical problems beyond those studied in the present paper. The generalization to spatially  short range perturbations of Minkowski
space will be presented in \cite{SZ19}; references to the literature in the ultra-static case can be found there.

We also hope to generalize the trace formula to vacuum black-hole  spacetimes,  such as the Schwarzschild- or Kerr-spacetimes. These are long-range perturbations of Minkowski space with many additional complications.  It would also be interesting  to
consider stationary spacetimes with timelike boundary as well as non-compact spacetimes. 

The   present article should be thought of as the first step towards a treatment of the spectral problem in the stationary setting. One of the main points is that virtually any spectral problem studied on Riemannian manifolds has a natural generalization to the stationary Lorentzian setting.

\subsection{Acknowledgements}  We are grateful to the Erwin Schr\"odinger institute in Vienna for hosting the programme ``The Modern Theory of the Wave Equation"
where this project started. The authors would also like to thanks Christian G\'erard and Colin Guillarmou for comments on a earlier version and useful discussions.

\section{The geometry of globally hyperbolic spacetimes}

By a spacetime we will mean a connected oriented time-oriented Lorentzian manifold $(M,g)$ of signature $(-,+,\ldots,+)$.
The {\it chronological future} $I_+(x)$ of a point is the set of
points that can be reached from $x$ by timelike future-oriented
curves. The {\it causal future} $J_+(x)$ is the set of points
reached from $x$ by causal curves (i.e. future directed curves whose tangent
vectors are timelike or lightlike). Given a set $A \subset M$,
$I_+(A) = \bigcup_{x \in A} I_+(x), J_+(A) : = \bigcup_{x \in A}
J_+(x).$ In general, $I_+(A) = \mbox{int} J_+(A)$ (the interior)
and $J_+(A) \subset \overline{I_+(A)}. $ There are similar
definitions for the past (using past oriented curves). For further
details, see \cite{HE} Chapter 6.

A smooth hypersurface $\Sigma \subset M$ is called {\it Cauchy surface} if every inextensible causal curve intersects $\Sigma$ exactly once.
A spacetime that admits a Cauchy surface is called {\it globally hyperbolic}. The class of globally hyperbolic spacetimes is the most natural class of spacetimes on which the initial value problem
for hyperbolic partial differential equations is well posed. Global hyperbolicity implies that $M$ is diffeomorphic to the product-manifold $\R \times \Sigma$. More precisely, by a result by Geroch \cite{G}
and Bernal-Sanches \cite{BS,BS2} there exists a smooth foliation by Cauchy hypersurfaces. On a general globally hyperbolic spacetime such a foliation by smooth Cauchy hypersurfaces is highly 
non-unique. From a physics point a view any particular choice splits the tangent space artificially into time- and spacelike directions and destroys relativistic covariance. From a more mathematical point of view
one would prefer to work with objects that are naturally associated with the category one works in (for example a suitable category of globally hyperbolic spacetimes).

Manifolds with a complete timelike Killing vector field are called {\it stationary}. If $(M,g)$ is a stationary globally hyperbolic spacetime then it is easy to see
(for example \cite[Lemma 3.3]{JS}) that \begin{equation}\label{St1}  (M,
g) \simeq (\R \times \Sigma, -(N^2-|\eta|^2_h) \dd t^2 + \dd t \otimes \eta + \eta \otimes \dd t +h),
\end{equation} 
where $(\Sigma,h)$ is a Riemannian manifold, $N: \Sigma \to \R_+$ is a positive smooth function, and $\eta$ a a co-vector field on $\Sigma$. In this case $\partial_t$ is a Killing vector field. Such stationary spacetimes are sometimes referred to as {\it standard stationary spacetimes}.
In case $\eta$ can be chosen to be zero such a stationary spacetime is called {\it static}. This means that the distribution defined by the orthogonal complement of the Killing vector field is integrable.
A very particular class of examples of globally hyperbolic spacetimes are product spacetimes $\R \times \Sigma$ with metric $g=-dt^2+h$, where $(\Sigma,h)$ is a complete Riemannian manifold. 
Spacetimes isometric to such products
are commonly referred to as {\it ultrastatic}. 
Hence, stationary spacetimes form a more general class than static spacetimes, and static spacetimes are more general than ultrastatic spacetimes.
Examples of stationary spacetimes that are in general non-static are Kerr and Kerr-Neuman spacetimes. The Schwarzschild spacetime is static but not ultrastatic.

It is known (see \cite{BGP} for references)
that a globally hyperbolic spacetime is isometric to a product
\begin{equation}\label{SPLIT}  (M, g) \simeq (\R \times \Sigma,  -N^2 \dd t^2 + h_t), \end{equation}
where $N: M \to \R_+$ is a positive smooth function and $h_t$ is a smooth family of metrics on $\Sigma$. Again, this representation is highly non-unique. If $(M,g)$ is stationary and globally hyperbolic
one cannot in general
choose this foliation compatible with the Killing flow in the sense that $\partial_t$ is the Killing vector field so that $h$ and $N$ do not depend on $t$.

A globally hyperbolic spacetime will be called {\it spatially compact} if there exists a compact Cauchy surface. In this case all Cauchy surfaces will be compact.


%

  \subsection{Symplectic geometry of the space of null-geodesics} \label{sec1.1}
 
\corri{In the following we review the symplectic and contact structures on the space of geodesics on a Lorentzian manifold. These structures have been subject to investigation in the context of Lorentzian geometry and we refer the reader to \cite{CN10,KT07,L01}}.

 Given a globally hyperbolic  spacetime $(M,g)$ we denote by $\ncal_p$ be the set of unparametrized inextensible lightlike geodesics. This set is the space of leaves of the null-foliation
 $\{(x, \xi) \in T^*M: g^{-1}_x( \xi,\xi) = 0 \}$. Here $g^{-1}$ denotes the induced metric on the cotangent space. 
 Similarly, we denote by $\ncal_a$ the space of affinely parametrised future-directed geodesics. If we identify affinely parametrised geodesics whose parameters $s',s$ are related by a simple shift
 $s'= s + c, c \in \R$ we obtain the space $\ncal$ of future-directed {\it scaled null-geodesics}. 
 Note that $\ncal$ carries an $\R_+$ action by rescaling the parameter. 
 
 By  Proposition \ref{symplmap} below, $\ncal$  is a symplectic manifold, and by definition it is a symplectic (Marsden-Weinstein) quotient. When $(M,g)$ is geodesically complete, the  quotient is defined as the symplectic manifold obtained by considering $\{(x,\xi) \in T^*M \setminus 0 \mid g^{-1}_x(\xi,\xi))=0\}$ and dividing out the $\R$-action induced by the Hamiltonian vector field. That way $\ncal$ is obtained by Hamiltonian reduction from the set of nullcovectors. This gives an  invariant characterisation of the symplectic structure \corri{(see \cite[Th 2.1]{KT07})}. 

 If $(M,g)$ is globally hyperbolic and $\Sigma$ a Cauchy surface then each element in $\ncal$ intersects $\Sigma$ exactly once. 
 The quotient exists as a manifold because the Cauchy surface provides us with a smooth cross section.
The  tangent vector of the geodesic is lightlike
 and, identifying $T^*M$ and $TM$ using the metric, so is the associated cotangent vector. The pull back of this co-vector to $\Sigma$ defines a covector in $T^*\Sigma \setminus 0$.
 This defines a map $\ncal$ to $T^*\Sigma \setminus 0$, which is invertible as for each element $\eta \in T^* \Sigma$ there is precisely one lightlike future directed covector $\xi \in T^*M \setminus 0$ 
 whose pull-back is $\eta$. This map is equivariant with respect to the $\R_+$-actions on $\ncal$ and $T^*\Sigma\setminus 0$.  
If the geodesic flow is complete on $M$ then it defines a Hamiltonian $\R$-action on $T^*M \setminus 0$.

 The following Proposition is closely related to results of Low \cite{L98, L01}.
 \begin{prop} \label{symplmap}
  If $(M,g)$ is a globally hyperbolic spacetime then the smooth structure and the symplectic structure on $\ncal$ do not depend on the Cauchy surface, and are therefore invariantly defined. Hence,
  for any Cauchy surface $\Sigma$ the above defined map
  $\ncal \to  T^* \Sigma \setminus 0$ is a homogeneous symplectic diffeomorphism.
 \end{prop}
 \begin{proof}
  Fix a Cauchy surface $\Sigma$. Let $T_{0,+}^* M$ be the fibre bundle of future directed nonzero null-covectors over $M$, and let $T_{0,+}^* M\vert_\Sigma$ be its restriction to $\Sigma$.
  As explained above we have a well defined  diffeomorphism $\nu: T^* \Sigma \setminus 0 \to T_{0,+}^* M\vert_\Sigma$. The latter is a smooth submanifold in $T^*M$.
  All together we have the following commutative diagram
  \begin{equation}
 \begin{tikzcd}\label{DIAGRAM-sympl}
   T^*\Sigma \setminus 0   \arrow[r,"\nu"]   \arrow[d,"\pi_1"] & T_{0,+}^* M\vert_\Sigma \arrow[hookrightarrow]{r}  \arrow[d,"\pi_2"] & T^* M \arrow[d,"\pi_3"] \setminus 0 \\
  \Sigma \arrow[r,"\mathrm{id}"]  & \Sigma  \arrow[hookrightarrow]{r}& M,
  \end{tikzcd}
 \end{equation}
 where $\pi_1,\pi_2,\pi_3$ are the fibre projections.  The map $T^* \Sigma \setminus 0 \to T_{0,+}^* M\vert_\Sigma, (x,\eta) \mapsto (x,\xi)$
 has the property that $\xi$ and $\eta$ project to the same co-vector on $\Sigma$, so they differ by an element in the conormal bundle. It follows that the pull-back of the tautological one-form
 on $T^*M$ to $T^*\Sigma$ is the tautological one-form on $T^*\Sigma$. Hence, the maps in the first line of the diagram are symplectic.  
 If $p \in \ncal$ and $\Sigma, \Sigma'$ are two Cauchy surfaces, then the corresponding points on $T^*\Sigma$ and $T^* \Sigma'$ are in the same orbit of the geodesic flow.
  The result now follows from the fact that the geodesic flow is Hamiltonian, with $H(\xi)=\frac{1}{2} g^{-1}(\xi,\xi)$, and therefore preserves the symplectic structure.
  \end{proof}
  
\begin{rem}
Note that the geodesic vector field on a spatially compact globally hyperbolic spacetime does not necessarily have to be complete. An example is given by the globally hyperbolic spacetime
$((-1,1) \times \Sigma, -\dd t^2 + h)$, where $(\Sigma,h)$ is a compact Riemannian manifold. 
\end{rem}

If $(M,g)$ is a spatially compact globally hyperbolic spacetime then $\ncal$ is a symplectic cone whose quotient, $\ncal_p$ by the $\R_+$-action is then a compact contact manifold.
For a given Cauchy surface $\Sigma$ this contact manifold is isomorphic to the projectivised cotangent bundle $\mathbb{P} (T^* \Sigma)$.

%
%

\subsection{Stationary spacetimes}

Assume now that $Z$ is a complete Killing vector field on the globally hyperbolic spacetimes $(M,g)$. Then the Killing flow $\Phi$ acts on the spaces $\ncal$ and $\ncal_p$.
Since the Killing flow is symplectic on $T^*M$ and commutes with the metric the induced flow on $\ncal$ is homogeneous and symplectic, and the flow on $\ncal_p$ is contact.
Our assumptions imply that the bicharacteristic flow is defined for all times.

\begin{prop}
 Assume that $(M,g)$ is a spatially compact stationary globally hyperbolic spacetime. Then $(M,g)$ is geodesically complete. 
\end{prop}
For a proof see for example \cite[Lemma 1.1]{A00}.
Note that we require the Killing vector field to be complete on a stationary spacetime. This is essential in our results.
If $(M,g)$ is a spatially compact stationary globally hyperbolic spacetime then $(M,g)$ isometric to a product $\R \times \Sigma$
with metric \eqref{St1}.
It is sometimes convenient to write this metric in the form
\begin{equation} \label{STANDARD}
  g =  -N^2 \dd t^2 + h_{ij}( \dd x^i + \beta^i \dd t)(dx^j + \beta^j \dd t),
\end{equation}
where $\beta$ (the shift vector field) is the vector field obtained from $\eta$ by identifying vectors with co-vectors using $h$ the above metric, i.e. $\beta^i = h^{ij} \eta_j$ in local coordinates.  
The coefficients are independent of $t$ so that the vector field
 $Z = \frac{\partial}{\partial t}$ is a timelike Killing vector field.

We would like to understand the periodic orbits of length $T \in \R$ of the flow. A periodic trajectory of length $T$ in $\ncal_p$ is then a maximal affinely parametrised geodesic 
$\gamma: \R \supset I \to \R \times \Sigma$ (in a standard stationary spacetime), $\gamma(s) =
(t(s), \gamma_\Sigma(s))$ such that the shifted geodesic $\gamma(s + T)$ is a reparametrisation of the original one.
In other words, $\gamma_\Sigma(s + T) = \gamma_\Sigma(s)$ for all $s \in I$. Obviously such a geodesic is defined on $\R$ and corresponds to a periodic curve on $\Sigma$.

\begin{example}
Lightlike $T$-periodic trajectories of Schwarzschild spacetime have
constant radial coordinate $r = 3 M$; other lightlike geodesics
escape to infinity or meet the boundary. Schwarzschild represents
empty space outside a non-rotating spherical massive body. Periodic trajectories of Kerr spacetime are
discussed in \cite{B}. One
might consider the exterior of a sphere in a spatially compact
spacetime as well. 
\end{example}

Note that in the ultrastatic case $(M,g) = (\R \times \Sigma, -\dd t^2+h)$ the curve $\gamma$ is a lightlike geodesic if and only if $\gamma_t$ is a unit speed geodesic on
the Riemannian manifold $\Sigma$. Hence, the set of periods $T$ of periodic trajectories on $\ncal_p$ is precisely the set of lengths of periodic geodesics of the Riemannian manifold $(\Sigma,h)$.
This is, by definition, the length spectrum of $(\Sigma,h)$.

\subsection{Residue symbol on $\ncal$ and the spatial volume of a stationary spacetime}

Assume that $(M,g)$ is a spatially compact globally hyperbolic spacetime of dimension $n$. Then $\ncal$ is a symplectic cone of dimension $2n-2$, and $\ncal_p \cong \ncal/\R_+$ is a compact contact manifold of dimension $2n-3$. 
On $\ncal$ we have the Euler vector field $D$ generating the $\R_+$-action and we have the symplectic volume form $\mathrm{dV}_{\ncal} = \frac{1}{(n-1)!}\omega^{n-1}$. Interior multiplication of this volume form
by $D$ gives a non-zero $(2n-3)$-form $\alpha = \iota_D \mathrm{dV}_{\ncal}$. This form satisfies $\mathcal{L}_D \alpha = (n-1) \alpha$ and hence is homogeneous of degree $n-1$. Suppose that $f \in C^\infty(\ncal)$
is homogeneous of degree $-(n-1)$. Then the form $f \alpha$ is homogeneous of degree $0$ and therefore is the pull-back of a unique $(2n-3)$-form $\alpha_f$ on $\ncal_p$. Integrating this form over $\ncal_p$
defines the symplectic residue of $f$.

\begin{defin}
 If $f \in C^\infty(\ncal)$ is homogeneous of degree $-(n-1)$ then the symplectic residue $\mathrm{res}(f)$ of $f$ is defined by $\mathrm{res}(f) = \int_{\ncal_p} \alpha_f$.
\end{defin}

The symplectic residue was introduced in \cite{Gu85} in the context of Weyl's law in algebras that quantize symplectic cones.
Now suppose that in addition $(M,g)$ is stationary. Then the Hamiltonian $H$ generating the Killing flow $\Psi^t_\ncal$ (see Lemma \ref{K}) is homogeneous of degree $1$ and everywhere positive.
Therefore the function $H^{-n+1}$ is homogeneous of degree $-n+1$ and we can form its symplectic residue $\mathrm{res}(H^{-n+1})$.
Using homogeneity of $H$ one arrives at
\begin{gather*}
 \mathrm{res}(H^{-n+1})  = (n-1) \mathrm{Vol}(\ncal_{H\leq 1}).
\end{gather*}

It is instructive to compute this number for a standard stationary spacetime with metric of the standard  form \eqref{STANDARD}.
Then, $|g|^{\frac{1}{2}} = N |h|^{\frac{1}{2}}$ and the inverse metric on the cotangent space takes the form
$$
 g^{-1} =  N^{-2}\left( \begin{matrix} -1 & \beta \\ \beta^T &  N^2 h^{-1} -\beta \otimes \beta \end{matrix} \right).
$$
If $\xi \in T^*\Sigma$ then the lightlike future directed lift has the form $(\beta(\xi) + N |\xi|_h) \dd t + \xi$.
Since the Killing field is $\partial_t$ the set $\mathcal{N}_{H \leq 1}$ is identified in $T^* \Sigma$ with the set
$$
 (\beta(\xi) + N |\xi|_h)  \leq 1.
$$
In an orthonormal basis in $T^*_p \Sigma$ with respect to the metric $h$ we can compute the fibre volume of this set. We choose the orthonormal basis so that $\beta(\xi_1)=|\beta|$
and $\beta(\xi_k)=0$ for $k \not=1$.
Denote $\kappa=N^{-1} |\beta|$. This is a function on $\Sigma$. The level set defined by $H=1$ is an ellipsoid satisfying the equation
$$
 (1 - \kappa^2) (\xi_1 + \frac{N^{-1} \kappa}{1-{\kappa}^2})^2 + \xi_2^2 + \ldots \xi_{n-1}^2 = \frac{1}{N^2(1- \kappa^2)}.
$$
Its volume is
$$
 N^{-n+1} \mathrm{Vol}({B}_{n-1}) {(1- \kappa^2)}^{-\frac{n}{2}} = N^{-n+1} \frac{1}{n-1} \mathrm{Vol}({S}_{n-2}) {(1- \kappa^2)}^{-\frac{n}{2}} 
$$
where $\mathrm{Vol}(B_{n-1})$ is the volume of the unit ball, and $\mathrm{Vol}(S_{n-2})$ the volume of the unit sphere in $\R^{n-1}$.
Hence, integrating over $M$ we obtain
\begin{gather}
  \mathrm{res}(H^{-n+1}) =  \mathrm{Vol}({S}_{n-2}) \int_\Sigma N(x) \left(N^2(x)-|\beta|_h^2(x)\right)^{-\frac{n}{2}} \mathrm{dVol}_h,
\end{gather}
and
\begin{gather}
  \mathrm{Vol}(\ncal_{H\leq 1})  =  \mathrm{Vol}({B}_{n-1}) \int_\Sigma N(x) \left(N^2(x)-|\beta|_h^2(x)\right)^{-\frac{n}{2}} \mathrm{dVol}_h.
\end{gather}

\section{The space of smooth solutions on  globally hyperbolic spacetimes}

Globally hyperbolic spacetimes are the most natural spacetimes for which the Cauchy problem for the wave operator is well-posed.
Suppose that $(M,g)$ is globally hyperbolic and $\Sigma$ is a smooth Cauchy surface. The Cauchy problem is to find a solution $u$ of $\Box u =0$
with given initial data $(u\vert_\Sigma, \nu_\Sigma u\vert_\Sigma)$ on $\Sigma$. Here $\nu_\Sigma$ denotes the future directed unit normal derivative at $\Sigma$.
There are various equivalent ways to show that this Cauchy problem has a unique solution in suitably defined function spaces. 
One way is via the construction of parametrices and fundamental solutions.
A continuous map $F: C^\infty_0(M) \to C^\infty(M)$ is called {\it fundamental solution} if $\Box F = F \Box = \mathrm{id}_{C^\infty_0(M)}$.
A {\it parametrix} is a map $F: C^\infty_0(M) \to C^\infty(M)$ such that $\Box F = F \Box = \mathrm{id}_{C^\infty_0(M)} \textrm{ mod } C^\infty$. This means a parametrix 
is an inverse modulo smoothing operators, i.e. operators whose integral kernels are in $C^\infty(M \times M)$. By the Schwartz kernel theorem $F$ has an integral
kernel in $\mathcal{D}'(M \times M)$ which we will also denote by the letter $F$ unless there is danger of confusion. If $f,g \in C^\infty_0(M)$ then 
$F(f \otimes g)$ will denote the pairing of the distributional kernel with the test function $f \otimes g \in C^\infty_0(M \times M)$, i.e.
$$
F(f \otimes g): = \int_M f(x) (F  g)(x) \mathrm{dVol}_g(x).
$$

Obviously every fundamental solution is also a parametrix.
A fundamental solution $E_\mathrm{ret/adv}$ is called {\it retarded/advanced} if $\supp(E_\mathrm{ret/adv} f) \subset J_\pm(\supp f)$.

The key theorem is the following (see for example Theorem 3.4.7 in \cite{BGP}).
\begin{theo}
 If $(M,g)$ is a globally hyperbolic spacetime, then there exist unique retarded fundamental solutions $E_\mathrm{ret/adv}$ for $\Box$. The difference $E=E_\mathrm{ret}-E_\mathrm{adv}$ gives rise to an exact sequence
 \begin{equation}
 \begin{tikzcd}\label{DIAGRAMexact}
  0 \arrow[r] &  C^\infty_0(M) \arrow[r,"\Box"]& C^\infty_0(M) \arrow[r,"E"] & C^\infty_{sc}(M)\arrow[r,"\Box"] & C^\infty_{sc}(M) 
 \end{tikzcd}
 \end{equation} 
 Here $C^\infty_{sc}(M)$ is the space of spacelike compactly supported smooth functions, i.e. functions whose support have compact intersection with any Cauchy surface.
\end{theo}

It follows from the uniqueness and formal self-adjointness of $\Box$ that the adjoint of $E_\mathrm{ret}$ is $E_\mathrm{adv}$ and vice versa.
This allows to extend the maps to the space of distributions 
$$
 E_\mathrm{ret/adv}: \mathcal{E}'(M) \to \mathcal{D}'(M).
$$
In particular $E$ is skew-adjoint.
The way these fundamental solutions are constructed is classical. 
Parametrices can be obtained in several ways. Constructions are due to  Hadamard \cite{H1,H2} and M. Riesz \cite{R1,R2}. Contemporary expositions include for example \cite{B,F,Gun, Be, BGP}. One can also construct fundamental solutions using global Fourier integral operator calculus. Using the fact that these parametrices are Fourier integral operators one obtains global information about the mapping properties of $E$. As before let $T_0^*M$ be the set of null covectors (the bundle of light-cones in cotangent space). Then $T_0^* M \setminus 0$ is a closed conic subset in $T^*M$
and the set
\begin{equation} \label{CDEF} C = \{(x,\xi,x',\xi')  \in (T_0^* M \setminus 0)^2 \mid (x,\xi) = G^t (x',\xi') \textrm{ for some } t \in \R\} \end{equation}
defines a homogeneous canonical relation from $T^*M \setminus 0$ to $T^*M \setminus 0$. Here $G^t$ denotes the geodesic flow.

\begin{theo}
 The map $E$ is a Fourier integral operator in $I^{-\frac{3}{2}}(M \times M, C')$.
\end{theo}
See \cite[Theorem 6.5.3]{DH}. 
Suppose that $\Sigma$ is a Cauchy surface in $M$. If $f \in C^\infty_0(\Sigma)$ we can define the distributions $\delta_{\Sigma,f}$ and
$\delta'_{\Sigma,f}$ 
by
\begin{equation} \label{distcauchy}
   \delta_{\Sigma,f}(\phi) = \int_\Sigma   f(x) \phi(x) \;\mathrm{dVol}_\Sigma(x), \quad  \delta'_{\Sigma,f}(\phi) =-\int_\Sigma f(x) (\nu_\Sigma \phi)(x) \;\mathrm{dVol}_\Sigma(x).
\end{equation}
Here $\nu_\Sigma$ is the future directed normal vector field to $\Sigma$. 
By the mapping properties of Fourier integral operators in $I^{-\frac{3}{2}}(M \times M, C')$ the function
\begin{equation} \label{SOLCP}
 u = E ( \delta'_{\Sigma,f_1} + \delta_{\Sigma,f_2})
\end{equation}
is well defined and smooth. 
Green's identity applied to the causal future of $\Sigma$ shows that $u$ is the unique solution of the Cauchy problem
$$
 \Box u = 0, \quad (f_1,f_2)=(u\vert_\Sigma, \nu_\Sigma u\vert_\Sigma).
$$
Indeed, by the support properties of retarded and advanced fundamental solutions we have $u = u_+ + u_-$, where $u_\pm = E_\mathrm{ret/adv}(\delta'_{\Sigma,f_1} + \delta_{\Sigma,f_2})$ is supported in $J^\pm(\Sigma)=:M_\pm$.
If $\phi \in C^\infty_0(M)$ is an arbitrary test function, then, 
\begin{gather*}
 (\delta'_{\Sigma,f_1} + \delta_{\Sigma,f_2}, \phi)= (u_+,\Box \phi)=\int_{M_+} (u \Box \phi) \mathrm{dVol}_g \\=
 \int_{M_+} (u \Box \phi - \phi \Box u) \mathrm{dVol}_g = \int_{\Sigma} - u(x)(\nu_{\Sigma} \phi)(x) + \phi(x) (\nu_{\Sigma} u )(x) \mathrm{dVol}_\Sigma(x).
\end{gather*}
Comparison with \eqref{distcauchy} shows that
$(u |_{\Sigma},  \partial_{\nu_{\Sigma}}u )  |_{\Sigma} = (f_1,f_2)$.
The space of smooth solutions of $\Box u=0$ is naturally a symplectic space, \corri{sometimes referred to as covariant phase space structure}, with symplectic form defined by
\begin{equation} \label{SYMPLECTIC}
 \sigma(u,v) = \int_{\Sigma} (\nu_x u)(x) v(x) - v(x) (\nu_x u) \mathrm{dVol_\Sigma},
\end{equation}
where integration is over any Cauchy surface $\Sigma$, and $\nu$ denotes the future directed unit normal vector field to $\Sigma$, i.e. $\nu_x = \nu_\Sigma(x)$.
Note that the definition does not depend on this choice of Cauchy surface as one can see immediately from Green's identity.
There is a close connection of this symplectic form to the 
propagator $E$. 
Integration  by parts shows that
$$
 E(f_1 \otimes f_2) = \sigma(u,v), \quad \textrm{if } u = E(f_1), v = E(f_2).
$$

Since the operators $\Box$ are real, the fundamental solutions $E_\mathrm{ret/adv}$ and the propagator $E$ commute with complex conjugation. They therefore
can also be viewed as operators acting on real-valued functions. 

\subsection{\label{kerBoxSect}Topologies and inner products on $\ker \Box$ for globally hyperbolic spacetimes}

As explained before there is in general no distinguished positive definite inner product on the space of solutions of $\Box u=0$. For general spatially compact globally hyperbolic spacetimes
there is however a distinguished topology of a Hilbert space, which we now describe.
We assume that $M$ is a spatially compact globally hyperbolic spacetime.
For the moment we fix a foliation $M= \R \times \Sigma$ by Cauchy surfaces. Let us denote $\Sigma_t:=\{t\} \times \Sigma$.
We also fix a compact slice $M_T := [-T,T] \times \Sigma \subset M$.
We define the finite energy space $\FE^s(M_T)$ to be
$$
 \FE^s(M_T):= C([-T,T],H^s(\Sigma)) \cap C^1([-T,T],H^{s-1}(\Sigma)).
$$
Similarly define 
$$
 \FE^s(M_T,\Box) := \{ u \in   \FE^s(M_T) \mid \Box u \in L^2([-T,T],H^{s-1}(\Sigma))\}
$$
equipped with the norm
$$
 \| u \| := \| u \|_{C([-T,T],H^s)} + \| u \|_{C^1([-T,T],H^{s-1})} + \| \Box u \|_{L^2([-T,T],H^{s-1})}.
$$
This space is a Banach space.

For each $t \in [-T,T]$ there is a natural map $R_t$ from $\FE^s(M_T)$ to $H^s(\Sigma) \oplus H^{s-1}(\Sigma)$
given by
$$
 R_t u = u|_\Sigma \oplus (\nu_\Sigma u)|_\Sigma,
$$
and similarly we have a continuous map
$$
 \FE^s(M_T,\Box) \to H^s(\Sigma) \oplus H^{s-1}(\Sigma) \oplus L^2([-T,T],H^{s-1}(\Sigma)), 
 \quad u \mapsto R_t \oplus \Box u.
$$
This map is a continuous bijection of Banach spaces. Since the right hand side has the topology of a Hilbert space,
so does the left hand side. A priory the space  $\FE^s(M_T,\Box )$ and its topology depend on the choice of foliation.
However, in the case $s=1$ the topology of the space $L^2([-T,T],H^{s-1}(\Sigma))=L^2(M_T)$ is independent of the chosen
foliation. Hence, the topology on $\FE^1(M_T,\Box)$ is invariantly defined.
In particular, the closed subspace
$$
 \FE^1(M_T,\Box) \cap \ker \Box
$$
has the topology of a Hilbert space that is defined independently of the chosen foliation.
It is also isomorphic to  $H^1(\Sigma_t) \oplus L^2(\Sigma_t)$ and therefore we can identify all the spaces
$\FE^1(M_T,\Box) \cap \ker \Box$ for $T>0$.   

The above estimates can be found  for example in the book \cite{T2} or in the more geometric setting \cite{BW} (see also \cite{Ho1}).

\begin{defin}
 The space $\ker \Box$ is defined to be $\FE^1(M_T,\Box) \cap \ker \Box$ equipped with its topology inherited from $H^1(\Sigma_t) \oplus L^2(\Sigma_t)$.
 As explained above this topology is that of a Hilbert space and is independent of the chosen foliation and $T>0$. We will denote by
 $\ker_\R \Box$ the set of real valued functions $\{ u \in \ker \Box \mid u=\overline u\}$ in $\ker \Box$.
\end{defin}

It is easy to see from the explicit integral representation that symplectic form, originally defined on the space of smooth solutions of $\ker \Box$, extends continuously to the space $\ker \Box$.


\subsection{\label{ESECT} Energy form for stationary spacetimes}

The inner products induced from the $H^1$ and $L^2$ inner products on a Cauchy surface $\Sigma$ depend on the Cauchy surface
and on a choice of an $H^1$-inner product on $\Sigma$. The above construction of a Hilbert space topology on $\FE^1 \cap \ker \Box$
therefore does not give an intrinsically defined Hilbert space structure.

Assume now that $(M,g)$ is a spatially compact stationary globally hyperbolic spacetime. In this case, if one drops the requirement of positive definiteness, 
one can give an intrinsically defined inner product on $\ker \Box$ using the stress energy tensor.
For $u \in C^\infty(M,\mathbb{R})$ let us define the stress-energy tensor $T(u)$ by the section in $\otimes_S T^*M$ given by
\begin{gather}
d u \otimes d u - \frac{1}{2} |du|^2 g -  \frac{1}{2} g V u^2.
\end{gather}
 Thus, $T$ is a symmetric $(0,2)$-tensor that is given in local coordinates (using Einstein's sum convention) by
$$
 T_{j k} = \nabla_j  u \nabla_k u - \frac{1}{2} g_{jk} \nabla_m u \nabla^m u -  \frac{1}{2} g_{jk} V u^2 .
$$
The divergence of $T$ satisfies
$$
 \nabla_j T^{jk} = -(\nabla^k  u)((\Box_g +V) u) -\frac{1}{2} u^2 \nabla^k (V).
$$
Hence, if $\Box = \Box_g + V$ and $\Box u = 0$, then
$$
 \nabla_j T^{jk} = -\frac{1}{2} u^2 \nabla^k (V).
$$
If $Z$ is a Killing vector field that commutes with $\Box$ this means that $Z V =0$ and therefore the covector field $T(u)(Z)$ is divergence free if $\Box u=0$. Indeed,
$$
 \nabla_j (T^{jk} Z_k) =  \frac{1}{2} \left( -|u|^2 Z (V) + T^{jk} (\nabla_j Z_k + \nabla_k Z_j) \right)=0.
$$

\begin{defin} The energy (quadratic) form of an element in $\ker \Box \cap C^\infty(M,\R)$ is defined by
$$Q(u) = \int_{\Sigma} \langle T(u)(Z), \nu \rangle \dd S$$
where $\nu$ is the unit normal to $\Sigma$, a spacelike
hypersurface and $Z$ is the timelike Killing vector field.
\end{defin}
This quadratic form is independent of the chosen Cauchy surface.
For a standard stationary spacetime with metric of the form \eqref{STANDARD}
a lengthy computation shows that
$$
 Q(u) = \frac{1}{2}\int_\Sigma \frac{1}{N} \left(  |\partial_t u|^2 + (N^2 h^{ij} - \beta^{i}\beta^{j}) (\partial_i u)(\partial_j u)  + V |u|^2 \right) \mathrm{dVol}_h.
$$
Therefore $Q(u)$ is actually defined on the space of real valued functions in $\ker \Box$ and it
can be extended and polarized to define a (possibly degenerate, possibly indefinite) hermitian form $Q(\cdot,\cdot)$ on $\ker \Box$.
\begin{lem}
 The energy quadratic form is invariant under the Killing flow, i.e.
 $$
  Q(e^{\rmi t D_Z} u,e^{\rmi t D_Z} u) = Q(u,u)
 $$
 for all $u \in \ker \Box$.
\end{lem}
\begin{proof}
 Since $Z$ is Killing we have $\mathcal{L}_Z Z =0, \mathcal{L}_Z g = 0$, $\mathcal{L}_Z V=0$, and $[\mathcal{L}_Z,*]=0$, where $*$ is the Hodge star operator.
 By the product rule that $Q(u,\mathcal{L}_Z u)+Q(\mathcal{L}_Z u,u)$ is the integral over the Lie-derivative of the closed $(n-1)$-form $* T_V(u)(Z)$.
 This Lie derivative is $\mathcal{L}_Z * T_V(u)(Z)$ equals the exact form $d T_V(u)(Z,Z)$. The integral of this exact form over $\Sigma$ vanishes by Stoke's theorem. 
 This shows that $\frac{d}{dt} Q(e^{\rmi t D_Z} u)=0$.
\end{proof}

Note that for every $(f_1,f) \in H^1(\Sigma) \oplus L^2(\Sigma)$ there exists a unique solution $u \in \ker \Box$ such that $u_t |_\Sigma=f_2$ and $u |_\Sigma = f_1$. 
Under this identification the quadratic form $Q$ is equivalent to the quadratic form $q=q_1 \oplus q_2$ on $H^1(\Sigma) \oplus L^2(\Sigma)$ defined by
\begin{gather}
q_1(f_1) = \frac{1}{2}\int_\Sigma \frac{1}{N} \left(   (N^2 h^{ij} - \beta^{i}\beta^{j}) (\nabla_i \overline f_1)(\nabla_j f_1)   + V |f_1|^2 \right) \mathrm{dVol}_h, \\
q_2(f_2) =  \frac{1}{2}\int_\Sigma \frac{1}{N}   |f_2|^2 \mathrm{dVol}_h.
\end{gather}
Note that $q_2$ is positive definite, whereas $q_1$ is the quadratic form associated to an elliptic self-adjoint operator $P$ with positive definite principal symbol.
Hence, the null space
$\ker Q = \{ u \in \ker \Box \mid Q(u,\cdot)=0$  is finite dimensional and is isomorphic to the kernel of the operator $P$. 

\begin{lem} \label{luckyzero}
 If $u \in \ker Q$, then $D_Z u =0$.
\end{lem}
\begin{proof}
 The above decomposition of $Q$ into direct summands $q_1$ and $q_2$ shows that if $u \in \ker Q$, then $u_t|_\Sigma=0$. But this must hold for any Cauchy surface $\Sigma$,
 so $\mathcal{L}_Z u =0$ everywhere.
\end{proof}

\subsection{Relation between the symplectic form and the energy form}

A direct computation shows that for $u,v \in \ker \Box$ we have
$$
 \sigma(D_Z u, v) = -\sigma(u, D_Z v). 
$$
In particular this implies that if $u$ and $v$ are (generalised) eigenvectors of $D_Z$ with eigenvalues $\mu$ and $\lambda$, then either $\sigma(u,v)=0$ or $\lambda = - \mu$.
The following Lemma is well known. 
\begin{lem} \label{Lemma12}
 Suppose $u,v \in \ker \Box$. Then
 $$
  Q(u,v)=\frac{\rmi}{2}  \sigma(\bar u, D_Z v) = \frac{1}{2} \sigma(\bar u, \mathcal{L}_Z v).
 $$
\end{lem}
\begin{proof}
 Note that $b(u,v))=\sigma(\bar u,\mathcal{L}_Z v)$ is a hermitian form that is obtained by extending the real quadratic form $b_\mathbb{R}(u,v))=\sigma(u,\mathcal{L}_Z v)$
 for real valued functions. We can therefore assume without loss of generality that $u,v$ are real-valued.
 By the polarization identity it is therefore sufficient to show that
 $2Q(u) = \sigma(u, \mathcal{L}_Z u)$ for any $u \in \ker \Box$. By continuity we can assume that $u \in \ker \Box \cap C^\infty$.
 Note that $\sigma(u, \mathcal{L}_Z u)$ is obtained by integrating the $n-1$-form $* \left( (\mathcal{L}_Z u) d u   -   u (d \mathcal{L}_Z u) \right)$.
 Now note that, using that $\mathcal{L}_Z$ commutes with the Hodge star operator $*$ and with the differential,
 \begin{gather*}
  * \left( (\mathcal{L}_Z u) d u   -    u (d \mathcal{L}_Z u) \right) = 2 * \left( (\mathcal{L}_Z u) d u \right) - \mathcal{L}_Z * (u \;du)=\\=
  2 * \left( (\mathcal{L}_Z u) du \right) - d \left( \iota_Z * (u \;du) \right) -  \iota_Z  d (u * du) = \\ =
  2 * \left( (\mathcal{L}_Z u) du \right) - d \left( \iota_Z * (u \;du) \right) -  \iota_Z \left( d u \wedge * du \right) + \iota_Z ( u * \delta du)=\\=
  2 * \left( (\mathcal{L}_Z u) du \right) - d \left( \iota_Z * (u \;du) \right) -  \iota_Z \left( d u \wedge * du \right) + \iota_Z  (u * \Box_g u).
   \end{gather*}
   The exact $n-1$-form $d \left( \iota_Z * (u \;du) \right)$ integrates to zero over any Cauchy surface and therefore we are left with
   \begin{gather*}
    \int_\Sigma * \left( (\mathcal{L}_Z u) du   -    u (d \mathcal{L}_Z u) \right) = \int_\Sigma \left( 2 d u(Z) d u (\nu) - g(\nu,Z) | du |_g - V |u|^2 \right)\mathrm{dVol}_{\Sigma}=\\=
    2 \int_\Sigma T_{jk} Z^j \nu^k \mathrm{dVol}_{\Sigma}.
   \end{gather*}
\end{proof}

\subsection{Krein space structure of $\ker \Box$ for stationary spacetimes}

Since the energy form is not necessarily positive definite a proper treatment of the spectrum of $D_Z$ requires a form of spectral theory of operators in Krein spaces.
A Krein-space is, by definition, an indefinite inner product space $(\mathcal{K},[\cdot,\cdot]) $ with the topology of a Hilbert space and the property that 
there exists a bounded operator $\mathcal{J}: \mathcal{K} \to \mathcal{K}$ with $\mathcal{J}^2=\mathrm{id}$ and such that $[\cdot,\mathcal{J}\cdot]$ is a Hilbert space inner product defining the topology. 
A choice of such a fundamental symmetry $\mathcal{J}$ gives a splitting
of the Krein space into positive and negative definite subspaces. A choice of $J$ and a choice of splitting is non-unique. If the dimension of the negative definite subspace with respect to any (and hence with respect to all) splitting is finite dimensional then the Krein space is called Pontryagin space.
Up to a certain point the theory of closed operators is parallel to the theory of closed operators in Hilbert spaces. The Krein-adjoint $A^+$ of a densely defined operator $A: \mathcal{K} \to \mathcal{K}$
can be defined as usual by $[A^+ \psi, \cdot] = [\psi, A\cdot]$ with domain consisting of these $\psi$ for which the right hand side extends to a bounded linear functional on $\mathcal{K}$.
An operator $A$ is called Krein-self-adjoint if $A^+=A$. 
An example of a Pontryagin space is given by $H^1(\Sigma)/\ker P$ with inner product $[u,v]_P = \langle u, P v \rangle_{L^2(\Sigma)}$, where $P$ is the self-adjoint elliptic differential operator with positive principal symbol introduced before on a compact Cauchy surface $\Sigma$. An example of a fundamental symmetry is the operator $J$ that commutes with $P$
and restricts to $-1$ on the non-positive eigenspaces of $P$ and to $+1$ on the positive ones. Since $P J = J P$ is still a second order elliptic pseudo-differential operator that differs from $P$ by a smoothing operator the intrinsic Krein space topology of $H^1(\Sigma)/\ker P$ is the quotient topology.
Keeping in mind the identification $\ker \Box$ with $H^1(\Sigma) \oplus L^2(\Sigma)$ and the explicit formula for the energy form we have shown:
\begin{prop} Suppose $(M,g)$ is a spatially compact globally hyperbolic stationary spacetime.
  Then the quotient space $\ker \Box / \ker Q$ with the energy form is a Pontryagin space. Its intrinsically defined topology coincides with the quotient topology.
  The dimension of this negative definite subspace equals the number of negative eigenvalues of the operator $P$ with multiplicities.
\end{prop}

 If $V \geq 0$ and $V(x)>0$ for some point $x \in M$ then $Q$ is positive definite. If $V=0$ then $Q$ is positive and has a one dimensional null space spanned by the constant function.

\section{\label{WTSECT} Wave-Trace on $\ker \Box$ for stationary spacetimes}

Since the operator $e^{\mathrm{i}D_Z t}$ commutes with $\Box$
and continuously maps Cauchy surfaces to Cauchy surfaces it restricts to a strongly continuous one parameter group $U(t)$
on $\ker \Box$. Since $e^{\mathrm{i}D_Z t}$ commutes with $\Box$ is also commutes with $E_\mathrm{ret}, E_\mathrm{adv},$ and $E$.

Recall that if $T$ is an operator acting on a Hilbert space then the property of $T$ being in the ideal of compact operators or in a Schatten ideal depends only on the topology of the Hilbert space and not on the inner product. Hence, the ideals of compact and trace-class operators on $\ker \Box$ are well-defined on $\ker \Box$. \corri{In particular the notions of trace-class operators and their trace do not depend on a choice of Cauchy surface.}

\begin{theo}\label{UphiFORM}
 Suppose that $\phi \in C^\infty_0(\R)$. Then the operator 
 $$
  U_\phi = \int_\R \phi(t) U(t) \dd t:   \ker \Box \to  \ker \Box
 $$
 is trace-class, and its trace equals
 $$
  \Tr(U_\phi)= \int_\Sigma \int_\R \phi(t) \left(  \nu_x \ee^{\mathrm{i}(D_Z)_x t}E(x,y) -  \nu_y \ee^{\mathrm{i}(D_Z)_x t}E(x,y)   \right) \dd t |_{y=x} \;\mathrm{dV}_\Sigma(x).
 $$
\end{theo}
Above, $\ker \Box$ is understood to be equipped with the norm of  $ \FE^1(M_T,\Box)$.

\begin{proof}
Fix a Cauchy surface $\Sigma$ to identify $\ker \Box$ with $H^1(\Sigma) \oplus L^2(\Sigma)$. Denote by $R$ the corresponding restriction map
$$
 R : \ker \Box \to H^1(\Sigma) \oplus L^2(\Sigma).
$$
One can describe the induced map $V(t) = R \circ U(t) \circ R^{-1}$ as a Fourier integral operator as follows.
The surface $\Sigma_t=\Phi_t \Sigma$ is again a Cauchy surface and we therefore obtain a foliation of a compact subset $M_T=\cup_{t \in [-T,T]} \Sigma$
of $M$. We can use this foliation to identify $M_T$ as a smooth manifold with the product $[-T,T] \times \Sigma$. This gives a global time coordinate $t$ on $M_T$
and the vector field $Z$ is given by $\partial_t$. The unit-normal to $\Sigma$ defines a vector field $\nu$ on $M_T$.
 The map $V(t)$ is then identified with the Cauchy evolution map
\begin{equation} \label{CDEVOLMAP}
 V(t) = R_t \circ R^{-1}.
\end{equation}
In $H^1(\Sigma) \oplus L^2(\Sigma)$ let $\mathrm{pr}_1$ and $\mathrm{pr}_2$ be the projections onto the components.
Then with 
$$
 V_\phi = \int_\R \phi(t) V(t) \dd t
$$
we have
$$
 \mathrm{pr}_1 V_\phi = I^*( \phi R^{-1}), \quad  \mathrm{pr}_2 V_\phi = I^*( \nu(\phi R^{-1}))
$$
Here $I_*$ is the push forward map that amounts to integration over $t$.
Since the wavefront set of $R^{-1}$ contains only light-like vectors in the second variable the push-forward results in a smooth kernel.
Hence, both maps $ I^*( \phi R^{-1})$ and $I^*( n(\phi R^{-1}))$ have smooth integral kernels and are therefore trace-class on all Sobolev spaces.

In the following consider the map $V_\phi: H^1(\Sigma) \oplus L^2(\Sigma) \to H^1(\Sigma) \oplus L^2(\Sigma)$ as a block matrix
$$
 V_\phi = \left( \begin{matrix} V_{11} & V_{12} \\ V_{21} & V_{22} \end{matrix} \right).
$$

Hence, $\Tr(V_\phi) = \Tr_{H^1}(V_{11}) + \Tr_{L^2}(V_{22})$. Now it is an easy observation that the trace on a smoothing operator is the same on every Sobolev space, in particular we have
$ \Tr_{H^1}(V_{11}) =  \Tr_{L^2}(V_{11})$.
We can use the fundamental solution $E$ to express the integral kernels of the maps $V_{11}$ and $V_{22}$ as follows.
\begin{gather*}
 V_{11}(x,y) =  -\int_\R   \phi(t) \ee^{\mathrm{i}(D_Z)_x t} \nu_y E(x,y) \dd t,\\
 V_{22}(x,y) =  \int_\R  \phi(t) \nu_x \ee^{\mathrm{i}(D_Z)_x t}E(x,y) \dd t.
 \end{gather*}
Integration over the diagonal yields the claimed form for the trace.
\end{proof}

\begin{rem}
 Instead of $\mathrm{FE}^1(M_T,\Box)$ we could also have used $\mathrm{FE}^s(M_T,\Box)$ for another $s \in \R$ to define the trace. Although the finite energy  spaces induce different topologies on the space of solutions the trace of a smoothing operator is independent of $s$.
\end{rem}

The above means that $\Tr(U(t))$ exists as a distribution in $\mathcal{D}'(\R)$ and is equal to
\begin{equation} \label{TRUtDEF}
 \Tr(U(t)) = \int_\Sigma \left(  \nu_x \ee^{\mathrm{i}(D_Z)_x t}E(x,y) -  \nu_y \ee^{\mathrm{i}(D_Z)_x t}E(x,y)   \right) |_{y=x} \;\mathrm{dV}_\Sigma(x).
\end{equation}
There is a more coordinate invariant way to write this expression, namely,
\begin{equation} \label{TRUtDEF2} \begin{array}{ll}
 \Tr(U(t)) =  \int_\Sigma * \left( \dd_x E_t(x,y) - \dd_y E_t(x,y) \right)|_{y=x},
\end{array} \end{equation}
where $*$ is the Hodge star operator on $M$ and where
\begin{equation} \label{EtDEF} E_t(x,y) = \ee^{\mathrm{i}(D_Z)_x t} E(x,y). \end{equation}
Note that the form $* \left( \dd_x E_t(x,y) - \dd_y E_t(x,y) \right)|_{y=x}$, with values in $\mathcal{D}'(\R)$, is closed as
\begin{gather*}
 \delta\left(  \dd_x E_t(x,y) - \dd_y E_t(x,y) \right)|_{y=x} = \left( \square_x E_t(x,y) - \square_y E_t(x,y)  \right)|_{y=x} \\= - \left( v(x) E_t(x,y) + v(y) E_t(x,y)  \right)|_{y=x} =0.
\end{gather*}
Therefore, it can be integrated over the cycle $\Sigma$.
Since all Cauchy surfaces are homologuous, the integral is independent of the chosen Cauchy surface.

Since $E$ is skew-symmetric and commutes with the flow
one obtains
$$
 E_t(x,y) = -E_{-t}(y,x).
$$
Hence, we also have
\begin{cor} The distributional trace defined above equals
\begin{equation} \label{TRUtDEF3}
 \Tr(U(t)) =  \int_\Sigma * \left( d_x (E_t(x,y) + E_{-t}(x,y)) \right) |_{y=x}.
 \end{equation}
\end{cor}

\section{Spectral Theory of $D_Z$}
 
Since the group $e^{\rmi t D_Z}$ leaves the null space of $Q$ invariant it also acts on the quotient space $\ker \Box / \ker Q$. It is a classical result by Naimark (\cite{N66}) that the analog of Stone's theorem about the correspondence between strongly continuous unitary one parameter groups and self-adjoint operators also holds for Krein spaces.
Hence, the operator $D_Z$ is naturally a Krein-self-adjoint operator on the Pontryagin space $\ker \Box / \ker Q$. The spectral theory for Krein-self-adjoint operators in Pontryagin spaces
(or more generally for definitizable Krein-self-adjoint operators in Krein spaces) is based in the existence of a spectral function with possible critical points  (see e.g. \cite{L82,DR96}). The treatment of Krein-selfadjoint operators in Pontryagin spaces can be simplified by the observation that given any self-adjoint operator $A$ on a Pontryagin space $\mathcal{K}$ there exists a decomposition $\mathcal{K}_0 \oplus \mathcal{K}_1$ into invariant subspaces $\mathcal{K}_0$ and $K_1$ such that
\begin{itemize}
 \item $\mathcal{K}_0$ a Pontryagin space,
 \item $\mathcal{K}_1$ is a Hilbert space,
 \item the restriction of $A$ to $\mathcal{K}_0$ is bounded and the restriction of $A$ to $\mathcal{K}_1$ is self-adjoint.
\end{itemize}
This decomposition is achieved by applying the spectral projectors avoiding the possible critical points (see for example \cite{LNT06}, Section 2, Subsection 5 for the detailed argument)

\begin{lem} \label{ftracelemma}
 Let $A$ be a Krein-selfadjoint operator in a Pontryagin space $\mathcal{K}$ and let $U(t)=e^{-\rmi A t}$ be the corresponding unitary group. Suppose that
 for each $\varphi \in C^\infty_0(\R)$ the operator
 $$
  U_\varphi := \int_\R \varphi(t) \ee^{-\rmi A t} \dd t
 $$
 is compact. Then $A$ has discrete spectrum consisting of eigenvalues $\lambda_j$ of finite algebraic multiplicity $m_j$. If $U_\varphi$ is trace-class, then
 $$
  \Tr U_\varphi = \sum_j m_j \hat \varphi(\lambda_j),
 $$
 where $\hat \varphi$ is the Fourier transform of $\varphi$. \corri{Moreover, in the decomposition $\mathcal{K}= \mathcal{K}_0 \oplus \mathcal{K}_1$
 above the space $\mathcal{K}_0$ is finite dimensional.}
\end{lem}
\begin{proof}
 As explained above $A$ can be written as a direct sum $A_0 \oplus A_1$ where $A_0$ is bounded Krein-selfadjoint operator on a Pontryagin space $\mathcal{K}_0$ and $A_1$ is a selfadjoint
 operator on a Hilbert space $\mathcal{K}_1$. For $A_0$ we have the Riesz-Dunford functional calculus and for $A_1$ the functional calculus for self-adjoint operators on Hilbert spaces. The Fourier transform $\hat \varphi$ of $\varphi$ is entire, and so we indeed have $U_\varphi = \hat \varphi (A_0) \oplus \hat \varphi (A_1)$. If $\varphi \in C^\infty_0(\R)$ with $\int_\R \varphi(x) dx =1$, let $\varphi_\epsilon:= \frac{1}{\epsilon} \varphi(\epsilon^{-1} x)$ be the corresponding $\delta$-family. Then $\hat \varphi_\epsilon(z) = \hat \varphi (\epsilon z)$ and this converges uniformly to one on compact sets of the complex plane. Hence, the identity map on  
 $\mathcal{K}_0$ is approximated in the operator norm by the compact operators $\hat \varphi_\epsilon(A)$. 
 Hence, it is compact, and so $\mathcal{K}_0$ is finite dimensional. Therefore, the spectrum of $A_0$ is discrete and $A_0$ can be brought to Jordan-normal form.
By the spectral mapping theorem we get $\mathrm{spec}(U_\varphi) = \hat \varphi(\mathrm{spec}(A)) = \hat \varphi(\mathrm{spec}(A_0)) \cup \hat \varphi(\mathrm{spec}(A_1))$.
If $\hat \varphi$ is real-valued then $U_\varphi|_{\mathcal{K}_1}$ is self-adjoint and compact. Therefore, there exists an orthonormal basis consisting of eigenvectors with eigenvalues $\mu_j(\varphi)$ and the only possible accumulation point of the spectrum of 
 $U_\varphi|_{\mathcal{K}_1}$ is zero. To show that the spectrum of $A_1$ is discrete we choose a point $\lambda$. It is easy to see that there is an even real-valued function $\varphi \in C^\infty_0(\R)$ such that $\hat \varphi(\lambda)>0$. The spectral mapping theorem then implies that the spectrum of $A_1$ near $\lambda$ is purely discrete of finite multiplicity. The statement about the trace follows from the spectral  theorem for $A_1$ and the fact that $A_0$ can be brought to Jordan-normal form.
 \end{proof}

 \begin{theo} \label{ponteigs} Suppose that $(M,g)$ is a spatially compact globally hyperbolic stationary spacetime. Then
 the spectrum of $D_Z$ on $\ker \Box$ is discrete and consists of at most finitely many non-real eigenvalues and infinitely many real eigenvalues that accumulate at $-\infty$ and $+\infty$.
 Moreover, the spectrum is invariant under complex conjugation $\lambda \to \overline{\lambda}$ and reflection $\lambda \to -\lambda$. There exist pairwise disjoint subsets 
  $B_-, B_0, B_+$  in $\ker \Box$ with spans $V_- =\mathrm{span}\; B_-, V_0 = \mathrm{span} \; B_0, V_+ =  \overline{\mathrm{span} \;B_+}$ such that
  \begin{itemize}
   \item $B_- \cup B_0 \cup B_+$ is linearly independent,
   \item $V_- ,V_0$ and $V_+$ are closed invariant subspaces for $e^{\rmi t D_Z}$,
   \item $V_- + V_0 + V_+ = \ker \Box$,
   \item $B_- \cup B_0 \cup B_+ \subset C^\infty(M)$,
   \item $\#B_-<\infty$, $\#B_0<\infty$, i.e. $\dim V_- < \infty$, $\dim V_0 < \infty$,
   \item the energy sesquilinear form $Q(\cdot,\cdot)$ is positive definite on $V_+$ and $B_+$ is an orthonormal subset. The null space $\ker Q$ is contained in $V_0$,
   \item $B_+$ consists of eigenvectors of $D_Z$ with real eigenvalues and $\overline B_+ = B_+$,
   \item $B_0$ consists of generalized eigenvectors of $D_Z$ with zero eigenvalues,
   \item $B_-$ consists of generalized eigenvectors of $D_Z$, i.e. if $u \in B_-$ then $(D_Z - \lambda)^m u =0$ for some $m \in \mathbb{N}$ and some $\lambda \in \C$.
  \end{itemize}
\end{theo}
\begin{proof}
 The spectrum is symmetric with respect to the reflection $\lambda \to -\overline{\lambda}$ because complex conjugation on $\ker \Box$ anti-commutes with $D_Z$. The symmetry with respect to complex conjugation
 follows from the fact that $D_Z$ is Krein-selfadjoint. Now recall that $\int \varphi(t) \ee^{\rmi t D_Z} \dd t$ is trace-class for any $\varphi \in C^\infty_0(\R)$. We can apply 
 Lemma \ref{ftracelemma} above to the Krein-self-adjoint operator $D_Z$ on the quotient space
  $\ker \Box / \ker Q$ to conclude that the spectrum of $D_Z$ is purely discrete. We decompose the Pontryagin space $\ker \Box / \ker Q$ as $\ker \Box / \ker Q = W_- \oplus W_+$, where $W_-$ is a finite dimensional invariant Pontryagin space and $W_+$ is an invariant Hilbert space. The operator $D_Z |_{W_-}$ admits a Jordan normal form and is therefore the span of generalised eigenvectors. The operator $D_Z |_{W_+}$ is self-adjoint with discrete spectrum and hence admits a complete diagonalisation.
 Given a generalised eigenvector $w \in \ker \Box / \ker Q$ of non-zero eigenvalue $\lambda$ with $(D_Z - \lambda)^N w=0$ and  $(D_Z - \lambda)^{N-1} w\not=0$ this means in $\ker \Box$
 that $(D_Z - \lambda)^N w = \tilde w$, where $\tilde w \in \ker Q$. By Lemma \ref{luckyzero} we have $D_Z \tilde w=0$ and hence $v = w - (-\lambda)^{-N} \tilde w$ satisfies
 $(D_Z - \lambda)^N v = 0$. Hence, any generalized eigenvector with non-zero eigenvalue in $\ker \Box / \ker Q$ has a representative that is an eigenvector with the same eigenvalue 
 in $\ker \Box$. This gives the set $B_-$ when applied to the generalised eigenvectors in $W_-$ and the set $B_+$ when applied to an orthonormal basis of eigenvectors in $W_+$. 
 By Lemma  \ref{luckyzero}, eigenvectors in $\ker \Box / \ker Q$ of zero eigenvalue simply correspond to generalised eigenvectors in $\ker \Box$.
 We choose a basis of generalised eigenvectors with zero eigenvalue in $\ker \Box / \ker Q$, lift it to a linearly independent set in  $\ker \Box$ 
 and extend it to a basis in $\ker \Box$. This way we obtain the set $B_0$ consisting of generalised eigenvectors with zero eigenvalue. 
 Since $D_Z$ anticommutes with complex conjugation the complex conjugate of the generalised eigenspaces with eigenvalue $\lambda$ is a generalised eigenspace with eigenvalue $-\lambda$. Since $B_0$ and $B_-$ are finite there are only finitely many possible eigenspaces in $V_+$ whose complex conjugate is not in $V_+$.
We can rearrange the sets $B_-$ and $B_+$ moving these finitely many eigenvectors to $B_-$, so $B_+$ can be chosen invariant under complex conjugation.
 The remaining claimed properties of this basis hold by construction.
  \end{proof}

 This theorem essentially says that $D_Z$ can be brought to Jordan normal form with only finitely many non-trivial Jordan blocks and all but finitely many eigenspaces are positive definite subspaces. We refer to this decomposition as the spectral decomposition of $D_Z$. It implies that the symplectic vector space $\ker \Box$ has a dense symplectic subspace that is a direct sum of finite dimensional symplectic vector spaces of the form $W_\lambda + W_{-\lambda}$, where
$W$ is the generalised eigenspace of $D_Z$ with eigenvalue $\lambda$. $W_0$ must therefore be even dimensional and symplectic.

This spectral theorem can now be used to more directly relate the  energy inner product $Q(u,v)$ of Section \ref{ESECT}  and symplectic 
inner product $\sigma(u,v) $ \eqref{SYMPLECTIC} on real eigenpaces of $D_Z$. 
It is convenient to introduce a subspace $W \subset \ker \Box$ as the closure of the span of all the real eigenspaces without non-trivial Jordan blocks, thus $W$ is an invariant subspace for
$D_Z$ on which $D_Z$ is completely diagonalisable.
In particular $W$ contains $V_+$ and has finite codimension in $\ker \Box$.
The form $Q$ may however not be positive definite on $V$. An important thing to note is that if $D_Z$ has only real eigenvalues and no non-trivial Jordan blocks in its spectral decomposition
then $W = \ker \Box$. 

\begin{theo}\label{JTHEO} Under the assumptions of Theorem \ref{ponteigs} let $W$ be as above. Then
there exists a complex structure $J: W \to W$ that commutes with $D_Z$ so that $\sigma(\overline v, J w)$ is a positive definite inner product on $W$. For any such complex structure 
the associated splitting $W = W^+ \oplus W^-$ of $W$ into $\pm \rmi$ eigenspaces of $J$ has the property that
$\pm D_Z$ is semi-bounded below on $W^\pm$.
\end{theo}
\begin{proof}
 If $W_\lambda$ is the eigenspace for eigenvalue $\lambda$
 then the symplectic form must be non-degenerate on $W_\lambda + W_{-\lambda}$. This follows from the non-degeneracy of $\sigma$ on $\ker \Box$ and the fact that if $v$ is in a generalised eigenspace with eigenvalue $\lambda$ then $\sigma(v,w)=0$ unless $w$ is in the generalised eigenspace with eigenvalue $-\lambda$.
  Since $\overline W_{\lambda} = W_{-\lambda}$ and by Lemma \ref{Lemma12} this implies that $Q$ is non-degenerate on $W_\lambda$ if $\lambda \not=0$. 
  It can therefore be decomposed into sign definite subspaces.
 This means we have a decomposition of $\ker \Box$ into subspaces $(W_k)_{k \in \mathbb{Z}}$ so that
 \begin{enumerate}
  \item $\oplus_{k \in \mathbb{Z}} W_k$ is dense in $\ker \Box$,
  \item each $W_k$ is spanned by eigenvectors with eigenvalues $\lambda_k$, where $\lambda_0=0$ and $\lambda_k \not=0$ if $k\not=0$,\\
  \item $\overline W_k = W_{-k}$,\\
  \item if $k \not=0$ then $Q$ is sign definite of sign $s_k \in \{-1+1\}$ on $W_k$.
   \end{enumerate}
 There are only finitely many $k$ such that $s_k = -1$. If $v$ is in $W_k$ then $\overline{v}$ is an eigenvector with eigenvalue $-\lambda_k$.
 Hence,
 $$W_0 \oplus \bigoplus _{s_k \lambda_k > 0} \left(W_k \oplus \overline{W_k} \right)$$ is dense. For a fixed $k>0$ with $s_k \lambda_k>0$ choose a basis $\{e_1,\ldots, e_m\}$ in $W_k$
 such that $\sigma(\overline{e_p},e_q) =- \rmi  \delta_{pq}$. This is possible since $\frac{\rmi}{2}\sigma(\overline{e_q}, D_Z e_q) = \lambda_k Q(e_q,e_q) >0$.
 Then $v_p=(e_p + \overline{e_p})$ and $w_p=-\rmi \;(e_p - \overline{e_p})$ are real-valued. We can define a linear map $J_k$ on
 $W_k \oplus \overline{W_k}$ by $J_k v_p = -w_p$ and $J_k w_p = v_p$. This means that $J_k e_p = \rmi \; e_p$, and $J_k \overline{e_p} = -\rmi \; \overline{e_p}$.
 Next choose any complex structure $J_0$ on the real part of $W_0$ such that $\sigma(J_0 u,J_0 v)=\sigma(u,v)$, $\sigma(u,J_0u) \geq 0$ for $u,v \in W_0$. Since this space is a finite dimensional symplectic space this is always possible.
 The map $J = \oplus_{k=0}^\infty J_k$ extends to the whole space $\ker_\R \Box$ by linearity and continuity. It also extends complex linearly to $\ker \Box$.
 We claim that
 \begin{gather*}
   \sigma(J u, J v) = \sigma(u,v),\\
   \sigma( \overline{u},J u ) \geq 0.
 \end{gather*}
 It suffices to check this in every space $W_k \oplus \overline{W_k}$ for $k\not=0$. The first equality follows from
 $\sigma(J \overline{e_p}, J e_p) = \rmi \; (-\rmi) \sigma(\overline{e_p}, e_p) = \sigma(\overline{e_p}, e_p) $. The second from
 $$
  \sigma(\overline{e_p}, J e_p) = \rmi \;  \sigma(\overline{e_p}, e_p) = 1
 $$
 and the corresponding complex conjugate $\sigma(e_p, J \overline{e_p}) = 1$. Since $\sigma$ is non-degenerate, so is the associated inner product. Hence
 $\sigma( \overline{u},J u ) >0$ if $u \not=0$.
 It remains to show that $\pm D_Z$ is semibounded below if $J$ is an invariant complex structure with these properties. 
Since $J$ commutes with the generator $D_Z$ the above decomposition into eigenspaces $W_k$ can be achieved in such a way that
 each $W_k$ is also an eigenspace of $J$ with eigenvalues $f_k \; \rmi$, where $f_k \in \{-1,+1\}$. Apart from finitely many $k$ we have $s_k=1$. In this case
 the $\sigma(e_p, J e_q) \geq 0$ leads to $f_k=1$ if $\lambda_k>0$ and $f_k=-1$ if $\lambda_k<0$. Therefore,
 $$
  \langle \overline{e_p}, D_Z e_p \rangle =\sigma(\overline{e_p},D_Z J_k e_p)= f_k \lambda_k  \geq 0.
 $$
 Therefore the spectrum of $\pm D_Z$ on $W^\pm$ is semi-bounded below. 
  \end{proof}
  
In the special case when $Q$ is positive definite on $\ker \Box$ we have $W = V_+ = \ker \Box$ and the above construction simplifies significantly. In this case 
$W^\pm$ are obtained by so-called frequency splitting. Namely, they are constructed as the positive and negative spectral subspaces of $D_Z$.
In the general case when Q is indefinite our proof shows that frequency splitting needs to be reversed on the negative definite subspaces to make sure the inner product is positive definite.
In that case $D_Z|_{W^+}$ may have finitely many negative eigenvalues.

Since $\Box$ and $Z$ are real the operator $D_Z$ has symmetric spectrum and the distribution trace is real-valued.
Classical trace formulae often deal with traces of the form $\Tr \ee^{\rmi t \Delta^{1/2}}$ where the generator of the group has positive spectrum.
The analog of this in our case would be the trace of the restriction $\Tr (\ee^{\rmi D_Z t}|_{W^+})$. The splitting of $W$ into $W^+$ and $W^-$
induces a splitting of the kernel of the operator 
$\ee^{\rmi D_Z t}$ into parts that are holomorphic in $t$ in the upper half plane and lower half plane respectively. This yields propagators with restricted wavefront sets.
Such splittings can always be achieved microlocally and are related to Hadamard states as explained in Section \ref{Hadamardsec}.

    \subsection{Formulation of the eigenvalue problem as an operator pencil}

As before assume that $(M,g)$ is a product $\R \times \Sigma$
with metric
$$
  g=-(N^2-|\eta|_h^2) \dd t^2 + \eta \otimes \dd t + \dd t \otimes \eta + h,
 $$
and inverse metric 
$$
 g^{-1} =  N^{-2}\left( \begin{matrix} -1 & \beta \\ \beta^T &  N^2 h^{-1} -\beta \otimes \beta \end{matrix} \right).
$$
Let $\tilde h$ be the metric obtained by inverting $N^2 h^{-1} - \beta \otimes \beta$. Then one has
$$
 |\tilde h|^{\frac{1}{2}} = N^{1-n} N |h|^{\frac{1}{2}}  (N^2-|\beta|_h^2)^{-\frac{1}{2}} =N^{1-n}  |g|^\frac{1}{2} (N^2-|\beta|_h^2)^{-\frac{1}{2}}.
$$
A longer computation shows that we have the decomposition
$$
 (N^2-|\beta|_h^2)^{\frac{1}{4}} N^\frac{n+1}{2} (\Box_g+V) N^{\frac{3-n}{2}} (N^2-|\beta|_h^2)^{-\frac{1}{4}} =\partial_t^2  - 2 X \partial_t - \Delta_{\tilde h} + W.
$$
Here $\Delta_{\tilde h}$ is the Laplace-type operator on the Riemannian manifold $(\Sigma,\tilde h)$, 
$W$ is the multiplication operator by the smooth potential
$$
 W = N^{-\frac{n-3}{2}}(N^2-|\beta|_h^2)^{-\frac{1}{4}} \left(\Delta_{\tilde h} ( N^{\frac{n-3}{2}} (N^2-|\beta|_h^2)^{\frac{1}{4}})\right) + N^2 V,
$$
and the skew-adjoint first order differential operator $X$ is given by
$$
  X= \frac{1}{2} \left( \mathcal{L}_\beta - \mathcal{L}_\beta^* \right).
$$

Denoting $P=- \Delta_{\tilde h} + W$ separation of variables then shows that $N^{\frac{n-3}{2}} (N^2-|\beta|_h^2)^{-\frac{1}{4}} \psi(x) \ee^{-i \lambda t}$  is an eigenfunction of $D_Z$ with eigenvalue $\lambda$
and is in $\mathrm{ker} \Box$ iff
\begin{equation} \label{pencil}
  ( P - 2 \I \lambda X - \lambda^2) \psi =0.
 \end{equation}
This is not of the form $A - \lambda^2$ and can therefore not be interpreted directly as an eigenvalue problem for an operator on $\Sigma$.
Families of operators depending as above on a parameter $\lambda$ are sometimes referred to as operator pencils. The spectrum of the pencil
is defined as the complement of the set of points $\lambda$ such that $( P - 2 \I \lambda X - \lambda^2)$ has a bounded inverse.
The eigenvalues are the values of $\lambda$ where $\mathrm{ker}(P - 2 \I \lambda X - \lambda^2)$ is non-trivial. In our case a simple application of the meromorphic Fredholm theorem
shows that $( P - 2 \I \lambda X - \lambda^2)$ is a meromorphic function with all negative Laurent coefficients of finite rank. This is another way of showing that
the eigenvalues form a discrete set.
Thus, the eigenvalues of the Killing vector field on $\ker \Box$ can be interpreted as the eigenvalues of a quadratic operator pencil.

It is well known \cite{G74, BN94, CN95, ER08} that the eigenvalue problem for a self-adjoint quadratic operator pencil of the above form \corri{can be transformed} into the self-adjoint generalised eigenvalue problem
$$
 A \psi = \mu B \psi, \quad \mu = \frac{1}{\lambda} - \lambda,
$$
where 
$$
 A = \left( \begin{matrix} 2 \rmi X & P-1 \\ P-1 & 2 \rmi X\end{matrix} \right), \quad B =\left( \begin{matrix} P & 0 \\ 0  & 1 \end{matrix} \right)
$$
are unbounded self-adjoint operators on $L^2(\Sigma) \oplus L^2(\Sigma)$. \corri{If $\lambda$ is a non-zero eigenvalue of the pencil then $\frac{1}{\lambda} - \lambda$ is a generalised eigenvalue with smooth generalised eigenvector.}
If $P$ is strictly positive then all the eigenvalues of the pencil are therefore real.
Note that there is no difficulty here with domains as such generalised eigenfunctions can be shown
to be smooth, using elliptic regularity. One can therefore consider this as a linear algebra eigenvalue problem on the space of smooth functions.

\section{Hadamard states and associated inner products}\label{Hadamardsec}

{In this section, we give another approach to inner products on $\ker \Box$
in terms of {\it Hadamard states} of quantum field theory. This discussion does not require a timelike Killing vector field and can be formulated on a general globally hyperbolic spacetime.
Hadamard states were first introduced in quantum field theory in curved spacetimes as a substitute for vacuum and positive energy states in the absence of a good definition of energy. We define such
states in classical PDE terms and only explain their relation to QFT at the
end of this section.}

\subsection{Inner products induced by Hadamard states}

Hadamard states come up very naturally by analysing the Fourier integral operator properties of $E$ on a globally hyperbolic spacetime $(M,g)$.
Note that the homogeneous canonical relation $C$ \eqref{CDEF} has two natural (closed and open) components $C_+$ and $C_-$
given by
\begin{equation} \label{Cpm}
 C_\pm = \{(x,\xi,x',\xi')  \in (T_{0,\pm}^* M)^2 \mid (x,\xi) = G^s (x',\xi') \textrm{ for some } s \in \R\}.
\end{equation}
These components are connected if $n >2$.
We can therefore write
$$
 E = S_+ - S_-,
$$
where
$$
 S_\pm \in I^{-\frac{3}{2}}(M \times M, C_\pm').
$$
Since this decomposition is non-unique and involves smooth cut-offs the integral kernels of $S_\pm$ are bi-solutions only modulo smoothing operators, i.e.
$$
 \Box S_\pm = S_\pm \Box = 0 \textrm{ mod } C^\infty.
$$
It was shown in \cite[Theorem 6.5.3]{DH} that this decomposition can always be chosen such that $\rmi S_-$ is non-negative as a bilinear form.
In fact, a stronger statement holds.

\begin{defin}
 A non-negative bi-distribution $\omega$ is called a {\it Hadamard state}  if $S_-:=-\rmi \omega$ and $S_+ = S_-^T$ define a microlocal splitting
 $E=S_+-S_-$ as described above and if in addition
 \begin{gather*}
 \Box S_\pm = S_\pm \Box = 0,\\
 S_+ = -\overline{S_-}.\\
\end{gather*} 
Here $S_\pm^T$ is the formally transpose map of $S_\pm$, i.e. we have $S_\pm(x,y)^T = S_\pm(y,x)$ for the integral kernels. 
\end{defin}
Hadamard states can be understood as arising from microlocal splittings that give rise to non-negative bisolutions.  Existence of Hadamard states was left open in \cite{DH}, and
it is a non-trivial fact that such states exist on any globally hyperbolic spacetime. The classical existence proof (\cite{FNW}, FSW78) is
based on a deformation argument  and propagation of singularities. Achieving a microlocal splitting in the analytic category, based on analytic wavefront sets, is much harder and a construction of analytic Hadamard states was given only recently \cite{GW17}. We refer to \cite{Rad96} for older definitions of Hadamard states and the relation to the microlocal splitting.

As explained in \cite[Section 6.6]{DH} the microlocal splitting  $E = S_+ - S_-$ is closely related to so-called distinguished parametrices. Let $\Delta_N$ be the conormal bundle of the diagonal in $M \times M$.
Recall that an orientation $(C_\nu^+,C_\nu^-)$ of $C \setminus \Delta_N$ is a decomposition $C \setminus \Delta_N = C_\nu^+ \cup C_\nu^-$ into two disjoint closed and open subsets $C_\nu^\pm$
such that the canonical relation defined by $C_\nu^+$ is the inverse canonical relation of $C_\nu^-$. For the operator $\Box$ on a connected globally hyperbolic spacetime 
of dimension greater equal three there are four orientations defined by
$C_\nu^+=C_\mathrm{F},C_{\mathrm{aF}},C_\mathrm{ret},C_\mathrm{adv}$, where
\begin{gather*}
 C_{\mathrm{F}} = \{ (x,\xi,x',\xi') \in C \mid (x,\xi) = G^s(x',\xi') \textrm{ for some } s \in \R_- \},\\
 C_{\mathrm{aF}} = \{ (x,\xi,x',\xi') \in C \mid (x,\xi) = G^s(x',\xi') \textrm{ for some } s \in \R_+ \},\\
 C_\mathrm{ret/adv} = \{ (x,\xi,x',\xi') \in C \mid x \in J_\pm(x') \}.
\end{gather*}
If the dimension is two there are more orientations, but even in this case we will use only the above ones.
For each orientation there exists a unique (modulo smoothing operators) parametrix $E_\nu$ such that $\mathrm{WF}'(E_\nu) \subset C_\nu^+ \cup \Delta_N$. In that case one show that
$\mathrm{WF}'(E_\nu) = C_\nu^+ \cup \Delta_N$.
Note that $E_{\mathrm{ret/adv}}$, the advanced and retarded fundamental solutions are distinguished parametrices with
$$
 \mathrm{WF}'(E_{\mathrm{ret/adv}})  \subset C_\mathrm{ret/adv} \cup \Delta_N.
$$
The distinguished parametrices $E_\mathrm{F}$ and $E_{\mathrm{aF}}$ corresponding to $C_\mathrm{F}$ and $C_{\mathrm{aF}}$ respectively are called the Feynman and anti-Feynman parametrices. 
If a Hadamard state $\omega = \rmi \, S_-$ is given, then one can define Feynman- and anti-Feynman
propagators $E_\mathrm{F}$ and $E_{\mathrm{aF}}$ that are are not only parametrices but actual fundamental solutions. They are given by
\begin{gather}
 E_{\mathrm{F}} = E_{\mathrm{adv}} - S_- = E_{\mathrm{adv}} + \rmi \;\omega  = E_{\mathrm{ret}}-S_+,\\
 E_{\mathrm{aF}} = \overline{E_\mathrm{F}} =  E_{\mathrm{adv}} + S_+ = E_{\mathrm{ret}} + S_-. 
\end{gather}
Some of these relations are given modulo smoothing operators in \cite[Section 6.6]{DH} with the notation $E_{\tilde N} = E_\mathrm{F}, E_\emptyset = E_{\mathrm{aF}}$, 
$S_{\pm} = S_{T^*_{0,\mp} M \setminus 0}$ , but they hold exactly if one constructs the objects from a Hadamard state.

\begin{defin} 
Given a Hadamard state $\omega=\rmi S_-$ one can define a hermitian form $(\cdot,\cdot)_\omega$ on the space $C^\infty_0(M)$ by
$$
 (f,g)_\omega := \rmi \left( S_-(\overline f \otimes g) + S_+(\overline f \otimes g) \right).
$$\end{defin}
Since both $S_-$ and $S_+$ are distributional bi-solutions this hermitian form descends to the quotient $C^\infty_0(M) / (\Box C^\infty_0(M))$.
By the exact sequence \eqref{DIAGRAMexact} the map $E$ defines an isomorphism from $C^\infty_0(M) / (\Box C^\infty_0(M))$ to the space of smooth solutions with compact space-like support.
We will denote this space by $\ker_{C^\infty_{sc}} \Box$. 
\begin{lem}
 For any Hadamard state $\omega$ the hermitian form $\langle \cdot,\cdot \rangle_\omega$ is a non-degenerate inner product on $\ker_{C^\infty_{sc}} \Box$.
\end{lem}
\begin{proof}
 We only need to show that $\langle f,f \rangle_\omega=0$ implies that $f=0$. This means we need to show that $(f,f)_\omega=0, f \in C^\infty_0(M)$ implies that $f \in \Box C^\infty_0(M)$.
 By assumption we have $ \rmi S_-(\overline f \otimes g)$ and  $\rmi S_+(\overline f \otimes g)$ are positive hermitian forms. Hence, $(f,f)_\omega$ implies that
 $\rmi S_\pm(\overline f \otimes f)=0$ and by the Cauchy-Schwarz inequality $\rmi S_\pm(\overline g \otimes f)=0$ for all $g \in C^\infty_0(M)$. This in turn implies that
 $E(\overline g \otimes f)=0$ for all $g \in C^\infty_0(M)$. This means that $f$ is in the kernel of $E$, and by the exact sequence \eqref{DIAGRAMexact}, $f \in \Box C^\infty_0(M)$.
\end{proof}

By the mapping properties of Fourier integral operators the inner product defined by a Hadamard state extends to $\ker \Box$. However, the topology induced by the inner product above is weaker than that of $\ker \Box$. In the case of a spatially compact globally hyperbolic spacetime 
one can show that it is independent of the Hadamard state and equals the topology induced by 
$\mathrm{FE}^{\frac{1}{2}}(M_T,\Box) \cap \ker \Box$ for any $T>0$. Indeed, if $(f,g) \in C^\infty(\Sigma)$ are the Cauchy data of a solution $u$, then, using the representation \eqref{SOLCP} one obtains
$$
\langle u, u \rangle_\omega = \sum_{\pm} \rmi \langle \delta'_{\Sigma,f} + \delta_{\Sigma,g}, S_\pm (\delta'_{\Sigma,f} + \delta_{\Sigma,g}) \rangle = \langle f, A f \rangle + \langle g, B g \rangle + 2 \langle f, C g \rangle,
$$
where $A,B$ are elliptic pseudodifferential operators on $\Sigma$ with positive principal symbols and orders $1$ and $-1$ respectively, and $C$ is a pseudodifferential operator of order $-1$. 
This is easily inferred from the principal symbols of $S_\pm$ as given in Theorem \ref{VOLC} and its proof. 


\subsection{Pure Hadamard states and complex structures}

Another way defining an inner product on the symplectic vector space $\ker_\R \Box$ is using a complex structure compatible with $\sigma$,
 i.e. a real linear map with $J^2=-\mathrm{id}$ that satisfies the additional conditions
\begin{gather}
 \sigma(J u, J v) = \sigma(u,v),\\
 \sigma(u, J u) \geq 0. 
\end{gather}
The Cauchy-Schwarz inequality together with non-degeneracy of $\sigma$ implies that $\sigma(u, J u) > 0$ if $u \not=0$.
One can endow $\ker_\R \Box$ with the structure of a complex vector space by defining $\rmi u := J u$, and 
\begin{gather}
 \langle u,v \rangle_J := \sigma(u,J v) + \rmi \sigma(u,v) 
\end{gather}
defines a positive definite inner product on $\ker_\R \Box$.
A quasifree state is then defined by the distribution $\omega_J \in \mathcal{D'(M \times M)}$  given by
\begin{gather}
 \omega_J(f_1,f_2) = \langle E(f_1),E(f_2) \rangle_J.
\end{gather}
It satisfies the following equations
\begin{gather}
 \omega_J(\Box f_1 \otimes f_2) =  \omega_J(f_1 \otimes \Box f_2) =0,\\
 \omega_J(f_1 \otimes f_2) - \omega_J(f_2 \otimes f_1) = -\rmi E(f_1,f_2),\\
 \omega_J(\overline f_1 \otimes f_1) \geq 0
\end{gather}
for any complex valued test functions $f_1,f_2 \in C^\infty_0(M)$.
The quasifree states obtained in this way are precisely the pure quasifree states on the CCR algebra (see e.g. Th 17.12 and Th 17.13, p. 427 in \cite{DG13}).
States obtained in this way are often called Fock states because their GNS representation is unitarily equivalent to a Fock representation.
We call a bi-distribution $\omega$ constructed in this way a Fock state associated with the complex structure $J$.
If in addition $\mathrm{WF}(\omega_J) \subset C_-'$ then $S_-=-\rmi \; \omega_J$ defines a splitting $E = S_+-S_-$ with $S_+=S_-^T$ and therefore $\omega_J$ is a Hadamard state.

\subsection{Relation of Hadamard state inner products and energy form
inner products} \label{Hadamard}

The following relates Theorem \ref{JTHEO} to Hadamard states.
\begin{theo}
 Suppose that $(M,g)$ is a spatially compact stationary spacetime, and let $\Box = \Box_g + V$, where $V \in C^\infty(M)$ commutes with the Killing vectorfield. Then 
 an invariant Hadamard state exists for $\Box$ if and only if the spectrum of $D_Z$ is real and there are no non-trivial Jordan blocks in the spectral decomposition of $D_Z$ on $\ker \Box$.
 In this case there exists an invariant Fock state. Such an invariant Fock state is automatically a Hadamard state.
\end{theo}
\begin{proof}
 If there exists an invariant Hadamard state this implies that the inner product $(\cdot, \cdot)_\omega$ is positive definite and invariant. Since all generalised eigenfunctions of $D_Z$
 on $\ker \Box$ are smooth they are also contained in this inner product space. Hence, all eigenvalues are real and there are no non-trivial Jordan blocks in the generalized eigenspaces.
Conversely, suppose all eigenvalues are real and there are no non-trivial Jordan blocks in the spectral decomposition of $D_Z$. Then, in the notation of Theorem \ref{JTHEO}, we have
$W = \ker \Box$. The complex structure $J$ constructed in Theorem \ref{JTHEO} gives rise to a Fock state such that the spectrum of $D_Z$ on the one particle Hilbert space $\mathcal{H}$ is semi-bounded below. This is sufficient to conclude that $\omega_J$ is a Hadamard state (Theorem 6.3 in \cite{SVW02}).
  \end{proof}

\begin{rem}
 Our definition of CCR algebra and Hadamard states is consistent with positivity of the generator of the time translation group. This means that for a quantum mechanical wave packet $\psi \in L^2(\R^{n-1})$ the time-dependent Schr\"odinger evolution is $\psi_t = \ee^{\rmi \, t H} \psi$, where $H$ is the Hamiltonian operator (the energy). With this convention positive energy solutions of the Schr\"odinger equation have their space-time Fourier transform supported in the half-space $[0,\infty) \times \R^{n-1}$. In physics it is more standard to use another sign convention. One can pass to the physics convention by taking the complex conjugate of the state and changing the commutator relation in the CCR algebra to $[\Phi(f_1),\Phi(f_2)] = \rmi E(f_1 \otimes f_2)$.
\end{rem}
\subsection{\label{QFT} Hadamard states in quantum field theory}
In this section we pause to recall the origin of Hadamard states in quantum 
field theory.
They are quasifree states on the CCR algebra of the generalised Klein-Gordon field defined by the operator $\Box$. The CCR algebra of $\Box$ can be defined as the abstract unital $*$-algebra defined by symbols $\Phi(f)$ indexed by $f \in C^\infty(M)$ and relations
\begin{gather}
 f \to \Phi(f) \textrm{ is complex linear},\\
 \Phi(f)^* = \Phi(\overline f),\\
 \Phi(\Box f) =0,\\
 [\Phi(f_1),\Phi(f_2)] = -\rmi E(f_1 \otimes f_2).
\end{gather}
A quasifree state is a state $\omega_H$  on this algebra that is completely determined by the functional 
$\omega: f_1 \otimes f_2 \mapsto \omega_H(\Phi(f_1) \Phi(f_2))$ such that the higher
components $\omega_H(\Phi(f_1) \cdot \Phi(f_n))$ can be expressed in a certain combinatorial way in terms of $\omega$. For a functional $\omega$
to define a quasifree state it needs to satisfy the conditions
\begin{gather}
 \omega(\overline f \otimes f ) \geq 0,\\
 \overline{\omega(f_1 \otimes f_2)}=\omega(\overline f_2 \otimes \overline f_1),\\
 \omega(Pf_1 \otimes f_2) = \omega(f_1 \otimes P f_2 )=0,\\
 \omega(f_1 \otimes f_2) - \omega(f_2 \otimes f_1) = -\rmi E(f_1 \otimes f_2). 
\end{gather}
If in addition $\omega$ is a bidistribution and $\mathrm{WF'}(\omega) \subset (T_{0,-}^* M)^2$ then $\omega$ is a Hadamard state. Here, as before,
$T_{0,\pm}^* M$ is the set of future/past directed nonzero null covectors in $T^*M$. 
We will not make use of the CCR algebra in this paper.

\section{Principal symbols of propagators}
Suppose $X$ and $Y$ are smooth manifolds of dimension $n_X$ and $n_Y$ respectively.
The principal symbol of a Fourier
integral distribution is a half-density taking values in the Maslov bundle. The Maslov bundle is a $\Z_4$-principal bundle and it appears because of additional factors of the form $\rmi^\sigma$ when changing phase functions
in the representation of the Fourier integral.
In particular any choice of local phase function defines a local trivialisation of the Maslov bundle. If a Fourier
integral distribution
$I^m \in \mathcal{D}'(X \times Y)$ given by
\begin{gather} 
I(x, y) = (2 \pi)^{-\frac{n_X + n_Y + 2 N}{4}} \int_{\R^N} \ee^{\rmi \phi (x, y, \theta) }
a( x, y, \theta) \dd \theta
\end{gather}
with non-degenerate homogeneous phase function $\phi$ and amplitude $a \in S^{m+\frac{n_X+n_Y-2N}{4}}_{cl}(X \times Y \times \R^N),$ the principal symbol is then
the transport to the Lagrangian $\Lambda_{\phi} =
 \iota_{\phi}(C_{\phi})$ of the half-density
 \begin{gather}
 a(\lambda) | \dd_{C_{\phi}}|^{\frac{1}{2}}: =   \frac{ a(\lambda) |\dd \lambda|^\half}{|D(\lambda,
\phi_{\theta}')/D(x,y,\theta)|^\half}
\end{gather} 
on $C_{\phi}$, where $\lambda=(\lambda_1,...,\lambda_n)$ are local coordinates on the critical manifold 
\begin{gather}
C_{\phi} =\{ (x,y,\theta); \dd_{\theta}\phi(x,y,\theta) = 0\},
\end{gather}
and where $\iota_{\phi}: C_{\phi} \to T^* (X \times Y) \setminus \{0\}$
is the map $(x, y, \theta) \to (x, \dd_x \phi, y, - \dd_y \phi).$ 
We refer to \cite[Section 25.1]{HoI-IV} for background.

In order to compute the principal symbol of $E \in I^{-\frac{3}{2}}(M \times M, C')$ note that by the decomposition
\begin{gather}
 E = S_+ - S_-, \quad S_\pm \in I^{-\frac{3}{2}}(M \times M, C'_{\pm})
\end{gather}
it is sufficient to compute the principal symbol of $S_\pm$. There is a natural half-density on the canonical relation  $C$ \eqref{CDEF}  (and hence on $C_\pm$ \eqref{Cpm}) constructed as follows.
For each point in $(x,\xi,x',\xi') \in C$ there exists a unique $s \in \R$ such that $(x,\xi) =  G^s(x',\xi')$. In that way, $C$ can be identified with an open subset 
of $(T^*_0M\setminus 0) \times \R$. We have $(T^*_0M\setminus 0) = T^*_{0,+} M\cup  T^*_{0,-}M$. The geodesic vector field, which is the Hamiltonian vector field
of $\frac{1}{2} g^{-1}$, defines a local flow on  $T^*_{0,\pm}M$ and the space of orbits is the symplectic manifold $\mathcal{N}$. The flow parameter $s$, the parameter $t$
and the symplectic volume form $\mathrm{dV}_{\ncal}=\frac{1}{d!} \omega^d$  on $\mathcal{N}$ then define a half-density 
\begin{equation} \label{dC} |\dd_C|^{\frac{1}{2}}=|\dd s'|^\frac{1}{2} \otimes |\dd s|^\frac{1}{2}  \otimes |\mathrm{dV}_{\ncal}|^\frac{1}{2}. \end{equation}
Here $d$ is half the dimension of $\mathcal{N}$ and $\omega$ is the symplectic form on $\omega$.
Note that if $\dim M = n, \dim T^* M = 2 n, \dim T^*M_{0,\pm}= 2n-1$ and 
$\dim C = 2 n$. Also, $\dim \ncal = 2 n- 2$ and $d=n-1$.

The principal symbol is a section of the half-density bundle of $C'$ times
a section of the Maslov bundle $M_{C'}$, where $C'$ is obtained from $C$
by changing the sign of $\eta$.\footnote{$M_C$ is denoted $L_C$ in  \cite{DH}.} The Maslov bundle is a flat complex line bundle associated to
a principal $\Z/4 \Z$ principal bundle. As in \cite[Section 6]{DG75}, the Maslov bundle of $C'$ is trivial and  has a global constant section. Indeed,
the subset with  $t =0$ is $N^*\Delta$ (the co-normal bundle of the diagonal)
and has a canonical constant section. It may be extended to a global
section of $M_{C'}$ by making it constant along null bicharacteristics. 

\begin{theo} \label{VOLC}
 The principal symbol of $S_\pm$ is the half-density $-\frac{\rmi}{2} \sqrt{2\pi} \cdot |\dd_C|^{1/2}$.
\end{theo}
\begin{rem}\label{VOLCREM}  In the relativistic setting, one defines the bicharacteristic flow as the Hamiltonian flow of $\half (\xi, \xi)_g$ rather than 
the Hamiltonian $|\xi'|_g$,  where $\xi'$ is the spacelike part of $\xi$,  in \cite{DG75}. The $\dd t$ in the relativistic setting corresponds to $|\xi'| \dd t$ in the non-relativistic setting. This explains
why one does not have a factor of $|\xi'|^{-1}$  in the product case corresponding
to $\Delta^{-\half}$  (see \eqref{productform}).  The principal symbol was computed by Duistermaat and H\"ormander \cite[Thm 6.6.1]{DH}. Their formula \corri{is different from ours by a factor} $\frac{1}{2}$ because their half-density was defined using the Hamiltonian
 flow generated by  $(\xi, \xi)_g$  rather than $\half (\xi, \xi)_g$. 
\end{rem}
\begin{proof}

 For the sake of completeness we give another proof here. Let $R_\Sigma$ be the restriction operator to a Cauchy surface $\Sigma$ and let $R^n_\Sigma$ be the operator that assigns to each
 function the future directed normal derivative on $\Sigma$. Then, we know that $R^n_\Sigma \circ E \circ (R_\Sigma)^* = \mathrm{id}$ and $R_\Sigma \circ E \circ (R_\Sigma)^* = 0$.
 The principal symbols of $S_\pm$ can be expressed as $a_\pm |d_C|^{1/2}$, where $a_\pm$ are functions on $C$. Since $E$ is a bi-solution of the wave equation in the sense that
 $\Box \circ E = E\circ \Box =0$ we are dealing with a product with vanishing principal symbol. \corri{Since the sub-principal symbol vanishes in our case the transport equation is $H_{p_k} a =0$ which implies that $a$ is a constant.}
 It follows that $a_\pm$ is constant along the Hamiltonian flow of the principal symbol of $\frac{1}{2} \Box$. Hence it is constant
 along the geodesic flow. Let $a^d_\pm$ be the restriction to the diagonal in $C_\pm$.
 Since the principal symbol of the future directed unit normal derivative is given by $\xi \to -\rmi \, n(\xi)$, where $n$ is the future directed normal vector-field one obtains from the decomposition
 $E = S_+ - S_-$ that $-(2 \pi)^{-\frac{1}{2}}(\rmi a^d_+ + \rmi a^d_-) = 1$ and $a_+^d-a_-^d =0$. This system has the unique solution $a^d_+=a^d_-=-\frac{\rmi}{2} \sqrt{2 \pi}$.
\end{proof}

\section{Proof of Theorem \ref{GUTZTH}}

\subsection{Translation by the Killing flow}

We first describe $E_t(x,y)$ as a Fourier integral operator.

\begin{theo}\label{COMPOKILLING} We have, 
$$E_t( x, y) \in I^{-\frac{7}{4}}(\R \times M \times M,  \ccal), $$ where
$$\ccal = \{(t, \tau, \zeta_1, \zeta_2) \in T^*(\R \times M \times M) \mid \tau + \zeta_1(Z)  = 0, (e^{t Z}(\zeta_1),  \zeta_2) \in C \}. $$

\end{theo}

The canonical relation is now parametrized by 
$$\R_t \times C \to \ccal, (t, \zeta_1,\zeta_2) \to (t, \zeta_1(Z),e^{t Z} (\zeta_1),\zeta_2).  $$
By Theorem \ref{VOLC}, the  principal symbol under the parametrization is given by,
\begin{lem} \label{SYMBOLFULL}
$$\sigma_{E_t}|_{C_\pm} = \mp \frac{\rmi}{2} (2 \pi)^{\frac{3}{4}} |\dd t|^{\half} \otimes | \dd_C|^{\half} .$$ 
\end{lem}
One can also derive this by using the properties of the restriction map to codimension one hypersurfaces as explained below.

\subsection{Completion of the proof of Theorem \ref{GUTZTH}}

The trace has now been defined in several equivalent ways, via. Lemma \ref{UphiFORM} or \eqref{TRUtDEF}-\eqref{TRUtDEF2}-\eqref{TRUtDEF3}.
It is now evident that
\begin{equation} \label{TRUtDEFpullback}
 \Tr(U(t)) = \pi_* \left(R_{\Sigma} \circ \Delta^* \left(  \nu_x E_t(x,y) -  \nu_y E_t(x,y)   \right)\right),
\end{equation}

where $\Delta(x) = (x,x)$ is the diagonal empedding and $\pi: \Sigma \times \R \to \R$ is the natural projection. 

The principal symbol of $d_x (E_t(x,y) + E_{-t}(y,x) )$ is computed using Lemma
\ref{SYMBOLFULL}.
The operators $R_{\Sigma}$ and $\Delta^*$ commute in the sense
that one can first restrict and then pull back to the diagonal or one can
first pull back and then restrict. The pull-backs to the diagonal have slightly different meanings depending on the order. 

We next need to work out the canonical relation and principal symbol of $R_\Sigma$. To do so, we review the canonical relation and
symbol of  restriction operators $\gamma_Y$  to 
a general codimension one hypersurface $Y$ in a general manifold $X$. The relevant symbol computations  can be 
found in \cite[Section 5.2]{TZ13}, and we only briefly summarize the results.

We define suitable Fermi normal coordinates  $(y, y_n)$ in a tubular neighborhood of $Y = \{y_n = 0\}$ and let $\sigma, \xi_n$ be the symplectic dual
coordinates.   The kernel of the restriction operator is
  \begin{equation} \label{OPAHpre}     \gamma_Y (y; y_n, y') = (2\pi)^{-n} \int \ee^{\rmi \langle y-y', \sigma \rangle  - i y_n \xi_n}  \, \dd \xi_n \dd\sigma.
  \end{equation}
  The phase $\phi(y, y_n, y', \xi_n, \sigma) =  \langle y - y', \sigma
   \rangle - y_n \xi_n  $   is linear and non-degenerate, the number  of phase variables is $N = n_X=n$ and  the dimension of $X \times Y$ is $2n-1$. Then $C_{\phi} =\{(y, x_n, y', \sigma,
   \xi_n): y = y', y_n = 0 \} $ and
   $\iota_{\phi}(y, 0, y, \sigma, \xi_n) \to (y, \sigma, y, \sigma,
0,\xi_n). $

The complication arises that elements of the form $(y, \xi, y, 0)$ appear when $\xi \in N^*Y$ in the canonical relation of $\gamma_Y$
and similarly $(y, 0, y, \xi)$ arises in that of $\gamma_Y^*$. Hence they are not homogeneous canonical relations in the sense
of \cite{HoI-IV}, i.e. conic  canonical relations $C \subset (T^*X \setminus 0) \times (T^* Y \setminus 0)$.  As in \cite{TZ13}, we  temporarily  introduce a cutoff  $(1-\chi)$ 
and let
 \begin{equation} \label{OPAH}     \gamma_{Y,\chi} (y; y_n, y') = (2\pi)^{-n} \int \ee^{\rmi \langle y-y', \sigma \rangle  - i y_n \xi_n}  (1-\chi(y',y_n,\sigma',\xi_n)) \, \dd \xi_n \dd\sigma.
  \end{equation}
so that no such elements occur in the support of the cutoff and then
\begin{equation} \label{gammaY} \gamma_{Y,\chi}  \in I^{ \frac{1}{4}}(X \times Y,
\Lambda_Y),  \end{equation}  where $\Lambda_Y =   \{(y, \xi, y, \sigma) \in T^*_Y X
\times T^* Y: \xi|_{TY} = \sigma\}.$ 

The phase is $\phi(y, y_n, y') = (y -y', \sigma) - y_n \xi_n$ with integration 
variables $\sigma, \xi_n$. 
The critical point equation is thus as before,

$$C_{\phi} = \{(y, y', \sigma, y_n, \xi_n): \dd_{\xi_n} \phi = y_n =  \dd_{\sigma} \phi = y-y' =0\}.$$
Since $(n_X + n_Y +2N)/4 = n- \frac{1}{4}$ the principal symbol near points where $\chi=0$ is $(2 \pi)^{-1/4} | \dd_{C_{\phi}} |^{1/2}$, where $\dd_{C_{\phi}}$ is the Leray measure on  $\{y_n = 0, y=y'\}$, that is   $ d_{C_{\phi}}: =   |\dd \xi_n \dd y' \dd \sigma|$.

\begin{lem} \label{RESLEM} Suppose that $A$ is a pseudodifferential operator microlocally supported away from $N^*\Sigma$. Then
$R_{\Sigma} \circ A: C_0^{\infty}(M) \to C^{\infty}(\Sigma)$ is a homogeneous Fourier integral operator with canonical relation $\Lambda_{\Sigma}$ and principal symbol given by the Liouville volume half-density 
$ (2 \pi)^{-\frac{1}{4}} \sigma_A(\mathrm{pr}_2(\xi)) |\dd \xi_n \dd y' \dd \sigma|^{\half}$.
Here $\mathrm{pr}_2$ is the projection $T^*M \times T^*M \to T^*M, (\xi_1,\xi_2) \mapsto \xi_2$.
\end{lem}

\corri{
The pullback operator $\rcal_{\Sigma}$ is not quite a homogeneous Fourier integral operator in the sense of \cite{HoI-IV} because 
the elements $(s, \xi, s, 0)$ occur in the wave front relation of $\rcal_{\Sigma}$  when $\xi \in N^*\Sigma$. Recall that a Fourier integral
operator is assumed to have a   homogeneous canonical relation
$C$, i.e. a wave front   $C \subset (T^*X \backslash 0) \times (T^* Y \backslash 0)$  disjoint from the zero section of the cotangent bundle. 
The purpose of this assumption is to ensure that the operator (and its dual) map smooth functions to smooth functions (rather than distributions); see \cite[(25.2.1)]{HoI-IV}.
By composing $\rcal_{\Sigma}$ with the cut-off operator $A$ as above,
we remove such `zeros in the canonical relation' and obtain a homogeneous Fourier integral operator of order $\frac{1}{4}$ with  principal symbol of $\rcal_{\Sigma}$  given by $(2 \pi)^{-\frac{1}{4}} |\dd \xi_n \dd y' \dd \sigma |^{\half}$ over $WF'(A)$ (the micro-support of $A$). This construction is also used to show that the restriction map is well defined on distributions whose wavefront set does not intersect the conormal.
More generally,  if  $B$ is a Fourier integral operator whose canonical relation
does not contain covectors of the form $(\eta,\xi)$ with $\eta \in N^*\Sigma$,  then  products of the form $R_{\Sigma} \circ B$ are   homogeneous Fourier integral operators. We will now use this fact without further mention.}

We now compute the canonical relation and principal symbol of $ \pi_* R_{\Sigma} \circ \Delta^*   \nu_x E_t(x,y) =
\pi_*  \Delta^*  \circ (R_{\Sigma} \times R_{\Sigma})  \nu_x E_t(x,y) $, where $R_{\Sigma}\times R_{\Sigma}$ is the restriction
$M \times M \to \Sigma \times \Sigma$. We recall the canonical relation $\ccal$ from  \ref{COMPOKILLING}, and 
$$\Lambda_{\Sigma} =  \{(\zeta |_{T \Sigma}, \zeta) \in T^* \Sigma \times T^*_{\Sigma} M\}. $$
Note that $ (R_{\Sigma} \times R_{\Sigma})  \nu_x E_t(x,y) $ is the restriction of the distribution  $\nu_x E_t(x,y)$, rather
than a composition of operator kernels. Below, we denote by $T^{* }_{0, \Sigma} M$ the lightlike covectors in $T^*_{\Sigma} M$.

\begin{lem} \label{CR1} Under the assumption of clean composition, $$(R_{\Sigma} \times R_{\Sigma})  \nu_x E_t(x,y) \in I^{-\frac{1}{4}}(\R \times \Sigma \times \Sigma, (\Lambda_{\Sigma}
\times \Lambda_{\Sigma}) \circ \ccal), $$where 
$$\begin{array}{l}   (\Lambda_{\Sigma}
\times \Lambda_{\Sigma}) \circ \ccal 
:= \{(t, \tau, (e^{t Z}\zeta_2) |_{T \Sigma}, \zeta_1 |_{T \Sigma}) \mid \tau = \zeta_1(Z), (\zeta_2, \zeta_1) \in C  \}.  \end{array}$$
Then, there exists a unique $s = s(\zeta_1, \zeta_2)$ such that $\zeta_2 = G^s(\zeta_1)$ and $(e^{t Z}\zeta_2)  \in T^*_{0,\Sigma} M$.
The  principal symbol of $(R_{\Sigma} \times R_{\Sigma})  \nu_x E_t(x,y)$  in the parametrization,
$\iota: \R \times T^{* } _{0,\Sigma} M \to (\Lambda_{\Sigma}
\times \Lambda_{\Sigma}) \circ \ccal $ defined by,
$$\iota(t, \zeta_1) = (t, \tau, e^{t Z} G^s (\zeta_1) |_{T \Sigma}, \zeta_1 |_{T \Sigma}), \;\;  (e^{t Z} G^s (\zeta_1) |_{T \Sigma}, \zeta_1 |_{T \Sigma}) \in T^*_{0, \Sigma} M\times T^*_{0, \Sigma}M. $$
With $\zeta_1 \in T_x^* M$ and $e^{t Z}G^s  \zeta_1 \in T^*_y M$ with $x, y \in \Sigma$, we have
 $$ (2 \pi)^{\frac{1}{4}} 
 \frac{|\langle \zeta_1, \nu_x \rangle|}{|\langle \zeta_1, \nu_x \rangle|^{\half} |\langle (e^{t Z}G^s \zeta_1), \nu_y \rangle|^{\half}} 
 |\dd t|^{\half} \otimes |\mathrm{d} \mu_{T^* _{0,\Sigma} M}|^{\half}.$$
 
\end{lem}
Above, $\mathrm{d} \mu_{T^* _{0,\Sigma} M}$ is the symplectic volume measure on $T^* _{0,\Sigma} M$, the lightcone bundle with
footpoint on $\Sigma$. The absolute value  $|\langle \zeta_1, \nu_x \rangle|$ is due to the change in sign of the symbols in Lemma \ref{SYMBOLFULL}. The assumption of clean composition is not needed for this statement to hold for $t$ in a neighborhood of zero.

 \begin{proof} 
The composition of canonical relations is immediate from Theorem \ref{COMPOKILLING} and \eqref{gammaY}. To compute the symbol we use
 Lemma \ref{SYMBOLFULL} and \eqref{dC},  keeping in mind Remark \ref{VOLCREM}.
  the operator with kernel $\nu_x E_t(x,y)$ has order $-\frac{3}{4}$ and principal symbol equal to $\frac{1}{2} (2 \pi)^{\frac{3}{4}} 
| \langle \nu_x, \xi \rangle| |\dd t|^{\half} \otimes |\dd_C|^{\half}$ on each component $C_+$ and $C_-$.  We note that if we identify
$\ncal$ with the cross-section $T^*_{0, \Sigma}M$, then   \eqref{dC}
is equivalent to  $$|\dd_C|^{\frac{1}{2}}=|\dd s'|^\frac{1}{2} \otimes |\dd s|^\frac{1}{2}  \otimes |\mathrm{dV}_{\ncal}|^\frac{1}{2} =
|\dd s'|^\frac{1}{2} \otimes |\dd s|^\frac{1}{2}  \otimes  |\mathrm{d} \mu_{T^* _{0,\Sigma} M}|^{\half}.$$

   We then compose the  symbol $ (2 \pi)^{-\frac{1}{4}}|\dd \xi_n \dd y \dd \sigma|^{\half} \otimes  (2 \pi)^{-\frac{1}{4}}|\dd \xi_n' \dd y' \dd \sigma'|^{\half}  $ of the  restriction (Lemma \ref{RESLEM})  to  $\R \times \Sigma \times \Sigma$ with the half-density $|d_C|^{\half}$. In a point of the composition, $(x, z) \circ (z, y) = (x,y)$,
   we refer to $(x,y)$ as output variables and $z$ as the input variable. Here, the points $(y, \sigma), (y', \sigma')$ are output variables 
   and $(s,s')$ and $(\xi_n, \xi_n')$ are the input variables. In symbol composition we obtain a density in the input variable and integrate over
   it to get a half density in the output variables.  The flowtime coordinate $s$ is related to the Fermi normal
   coordinates $(y_n, \xi_n)$ by $ds = \frac{dy_n}{|\xi_n|}$. Making this substitution in the half-densities, the symplectic pairing
   between $dy_n$ and $d \xi_n$ eliminates this pair in the composition, leaving $$\frac{1}{|\xi_n|} \frac{1}{|\xi_n'|} 
    |\mathrm{d} \mu_{T^* _{0,\Sigma} M}|^{\half}
   = \frac{1}{|\langle \zeta_1, \nu_x \rangle|^{\half} |\langle (e^{t Z}G^s \zeta_1), \nu_y \rangle|^{\half}} 
 |\mathrm{d} \mu_{T^* _{0,\Sigma} M}|^{\half}, $$
 where we evaluate $\xi_n, \xi_n'$ at the relevant points.
   Multiplying by $|\langle \zeta_1, \nu_x \rangle|$, the universal constant, and tensoring with $|dt|^{\half}$ completes the calculation of the
   symbol.

\end{proof}

Next, we pullback under the diagonal embedding $\Delta$ and pushforward under $\pi$  to obtain the canonical relation and symbol
or \eqref{TRUtDEF}. 
 Recall that by Proposition \ref{symplmap}
we have a well defined symplectic diffeomorphism $\phi: \mathcal{N} \to T^* \Sigma \setminus 0$.

\begin{lem} \label{CR2} The canonical relation of \eqref{TRUtDEF} is given by, \begin{equation} \label{PF1}\begin{array}{lll}WF (\pi_* (R_{\Sigma} \Delta^* \nu_x E_t(x,y) )) & \subseteq &
 \{(t, \tau): \exists \gamma, \ee^{t Z} \gamma = \gamma \in \ncal, \tau =  H(\gamma)\}.
 \end{array} \end{equation}
 The principal symbol  of the distribution $\Tr U(t)$ at $t=0$ is given by 
$$
 2 (n-1) \mathrm{Vol}(\mathcal{N}_{H \leq 1}) (2 \pi)^{-n +1}\mu_{n-1}(t),
 $$
 where $\mu_{n-1}(t)$ is the distribution defined by the oscillatory integral
$$
 \mu_{n-1}(t) = \frac{1}{2}\int_{-\infty}^\infty \mathrm{e}^{-\rmi \tau t} |\tau|^{n-2} \mathrm{d}\tau.
$$

\end{lem}

\begin{proof}
We  pull back the canonical relation  $\Lambda_{\Sigma} \circ \ccal $ of Lemma \ref{CR1}  to the diagonal to get,
\begin{equation} \label{CRRDELTA} \begin{array}{lll}WF (R_{\Sigma} \Delta^* d_x E_t(x,y) & \subseteq &
\{(t, \tau,  (\phi \ee^{t Z} \phi^{-1} \eta)  - \eta))  \mid  \tau = \eta(Z) , \eta \in T^* \Sigma \setminus 0 \}. 
\end{array}\end{equation}
The  pushforward of this Lagrangian manifold under $\pi$ is then
\begin{equation} \label{PF}\begin{array}{lll}WF (\pi_* (R_{\Sigma} \Delta^* d_x E_t(x,y) )) & \subseteq &
 \{(t, \tau): \exists \gamma, \ee^{t Z} \gamma = \gamma \in \ncal, \tau =  H(\gamma)\}.
 \end{array} \end{equation}

 This proves the first statement of Theorem \ref{GUTZTH}, namely that the singular times are a subset of the periods of
 $e^{t Z}$ acting on $\ncal.$ The next and main step is to compute the principal symbol of the trace at each period.
 When we pull back to the diagonal, $e^{t Z}G^s( \zeta_1) =\zeta_1$, and then 
   $$ \frac{\langle \zeta_1, \nu_x \rangle}{|\langle \zeta_1, \nu_x \rangle|^{\half} |\langle e^{t Z}G^s (\zeta_1), \nu_y \rangle|^{\half}} 
= 1. $$

 We begin with the symbol at $t=0$.
 The principal symbol is a homogeneous half-density on $T^*_t \R$ and therefore may be represented as a constant
 multiple of $|\dd\tau|^{\half}$ on each half-line $T_+^* \R,$ resp. $T^*_- \R$.

  Restriction to the diagonal and integration over $\Sigma$
 gives an element in $I^{(n-1)-3/4}(\R)$ with principal symbol at $T^*_0 \R$ given by $\mathrm{res}(H^{-n+1}) (2 \pi)^{3/4-(n-1)} |\tau|^{n-2} |\dd \tau|^{1/2}$. Here homogeneity of the principal symbol was used and additional factors appear due to an excess $e=n-2$ (see for example
\cite[Prop. 25.1.5']{HoI-IV}) of the phase function in the representation of the distribution.  

\end{proof}
 
\subsubsection{Principal symbol at $T \in \pcal$} To complete the proof of  Theorem \ref{GUTZTH},  we need to compute the principal symbol
at periods $T \in \pcal$.

\begin{lem} \label{SYMBOLTINP} If the closed orbits are non-degenerate, then the symbol of \eqref{TRUtDEF}  at $t = T \in \pcal$ is
the half-density on  $T_{T}^*\R$  given by,
$$\frac{\ee^{-\frac{i\pi}{2} m_{\beta}}}{ \sqrt{|\det(I - P)|}} |\dd \tau|^{\half}$$
where $\tau$ is the linear coordinate on $T_T^*\R$.
\end{lem}

\begin{proof} 
The principal symbol at a period $t \in \pcal$ is computed by the Duistermaat-Guillemin formula \cite[Lemma 4.3-Lemma 4.4]{DG75}. The calculation
can be (and has been) formalized as a composition law for `weighted canonical relations', i.e. canonical relations equipped with half-densities.
Here, one ignores the Maslov index for expository simplicity.  

The first of the  weighted canonical relations one is composing is the spacetime graph \eqref{CRRDELTA} of the Killing flow on $\ncal$, together
with the half density $ R_{\Sigma}^*\Delta^* \sigma_{E_s}$ where $\sigma_{E_s}$ is given in Lemma \ref{SYMBOLFULL}. The second is
the canonical relation of $\pi_*$ together with its half-density symbol.  The composition of the canonical relations is given by \eqref{PF}. 
To compose the half-density symbols, let us review the definition.
Suppose that $X, Y$ are compact manifolds and $\Gamma \subset T^*(X \times Y) \setminus 0$, $\Lambda \subset T^* Y \setminus 0$
are Lagrangian submanifolds. Let $\Gamma' = \{(x, \xi, y, \eta) \colon (x, \xi, y, - \eta) \in \Gamma\}$. The composition of $\Gamma $ and $\Lambda$ is
the Lagrangian submanifold 
defined by
$$\Gamma' \circ \Lambda = \{(x, \xi) \colon \text{there exists $(x, \xi, y, \eta) \in \Gamma'$ where
$(y, \eta) \in \Lambda$}\}. $$
 Moreover if $F \subset \Gamma' \times \Lambda$ is
the fiber product, i.e.  set of points $((x, \xi, y, \eta), (y, \eta))$, and if the
fibers are compact, the half-densities
compose on each tangent space to give a density on the fiber with
values in half-densities on the composition: Let $q \in \Gamma \circ \Lambda$. Let $F_q$ be the fiber
over $q$ and let $m = (m_1, m_2, m_2) \in F_q$. Then,
\begin{equation}\label{ISO}  |T_{(m_1, m_2)} \Gamma|^{\half}
\otimes |T_{m_2} \Lambda|^{\half} \simeq |T_m F_q| \otimes |T_{m_1, m_2} \Gamma \circ
T_{m_2} \Lambda|^{\half} . \end{equation}

Following \cite{HoI-IV} Theorem 25.2.3, let us  denote the half-density on $\Gamma$ by $\sigma_2$ and the
half-density on $\Lambda$ by $\sigma_1$ and let $\sigma_1 \times \sigma_2$ denote the density on $T_f F$ with values in half-densities on
on $ T_{m_1, m_2} \Gamma \circ
T_{m_2} \Lambda $. We review the proof in Appendix \ref{SYMBOLAPP}. Then integration over $F$ 
gives a half-density on the composite. At a point $q \in \Gamma \circ \Lambda, $
\begin{equation} \label{COMPO} \sigma_1 \circ \sigma_2 |_{q}= 
\int_{F_{q}} \sigma_1 \times \sigma_2. \end{equation}

For the compositon in \eqref{PF}, the fiber over $(t, \tau) \in T^*_t \R$ is the set of points $$F_t = \{\gamma \in \ncal:
e^{t Z} \gamma = \gamma\}.$$
The fixed point sets of $e^{t Z}$ on $\ncal$  consist of the union of periodic orbits $\beta$ of the Killing flow on $\ncal$. They  vary considerably as one varies $(M,g)$,  and we restrict to the case where
the periodic orbits $\beta$ are non-degenerate. Let us define the term.  Since $e^{t Z}$ is a Hamiltonian flow on $\ncal$ with Hamiltonian $H(\gamma) = \langle Z, \zeta)$
for $\zeta \in \gamma$, the flow preserves the level sets $\{H(\gamma) =  \tau\}$. Let
$\gamma \in \beta$ and let $\acal:= D_{\gamma} \ee^{t Z}:  T_{\gamma} \{H =  \tau\} \to T_{\gamma} \{H =  \tau\}$.   Let $V =  T_{\gamma} \{H(\gamma)= E\}$.
The restriction of the symplectic form  $\Omega$ to $V: = T_{\gamma} \{H = E\}$ (for any $E$) satisfies $\Omega(\xi_H, \cdot) = 0$,  and  non-degeneracy of $\beta$ means that  the  nullspace of $\Omega |_V$  is spanned by $\xi_H$. Also,
$(I - \acal) \xi_H = 0$ and non-degeneracy implies that $\ker (I - \acal) = \R \xi_H$, where $\ker (1 - \acal)$ is the kernel of $(I - \acal)$ on $V$. Hence $\Omega $ is a symplectic form
on $V/\ker(I - \acal)$.
Moreover,        $v \in \ker (I -\acal)$ is symplectically orthogonal to the range
$\rm{Im}(I - \acal) $  of
$(I - \acal)$ on $V$, i.e. $\ker (1- \acal) \subset  \rm{Im} (I - \acal)^{\perp}$. Non-degeneracy implies that ${\rm Im}(I - \acal) = V/\R \xi_H$,
so that $(I - \acal)V$ is a symplectic subspace of $(V, \Omega)$. Moreover, $V/\rm{Im}(I - \acal) = \R \xi_H$.
We define $(I - P): V/\ker  \to V/\ker $ to be the linear symplectic  quotient map induced by $I - \acal$.\footnote{Our notation differs from that of \cite{DG75}, where $V$ is the symplectic vector space $T_{\gamma} \ncal$.}

Returning to symbol composition, we claim that the integrand of \eqref{COMPO} is given on each component $\beta$ of
the set of periodic orbits of period $t$  by
\begin{equation} \label{DET(I-P)} \sigma_1 \times \sigma_2 = \frac{\ee^{-\frac{i\pi}{2} m_{\beta}}}{ \sqrt{|\det(I - P)|}} \dd t \otimes |\dd \tau|^{\half}, \end{equation}
where $\sigma_2$ is given by Lemma \ref{SYMBOLFULL},where $\sigma_1$ is the principal symbol of $\pi_* \Delta^*$ and where $\dd t$
is the natural 1-form on $\beta$ (given its parametrization). This follows by tracing through the isomorphism of \eqref{ISO} and of course
is done in generality in \cite{DG75}. One of the key points is that the principal symbol of $\pi_* \Delta_*$ is the symplectic volume $\half$-density 
on the diagonal (\cite[Lemma 6.3]{DG75}). Hence both of the half-densities being composed in the trace $\rm{Tr}(U(t))$ are volume
half-densities, one on the diagonal and one on the graph of $P$. In this case the exact sequence \eqref{5.2} 
\cite[(5.2)]{DG75}  is equivalent to the exact sequence \eqref{4.1} \cite[(4.1)]{DG75}. The composite half-density is then calculated
by tracing through the isomorphisms in \eqref{4.1} and as in \cite{DG75} it gives \eqref{DET(I-P)},at least up to the Maslov factor. The latter is calculated as in \cite[Section 6]{DG75} or \cite[Section 6]{DH}.

\end{proof}

\begin{rem} As pointed out in \cite{R91,M94}, the Maslov index can be identified with the Conley-Zehnder index of the periodic orbit, which is
manifestly independent of the choice of coordinates and is an intrinsic
invariant of the orbit.

\end{rem}
 
\section{Novelty of the method}

In this section, we use the operator pencil formulation to show that
our trace formula is not a simple consequence of separating variables and
employing the Duistermaat-Guillemin trace formula.

If $X$ commutes with $P$ then the pencil \eqref{pencil} can be factorised into linear scalar factors using spectral calculus. If they do not commute one might hope that one can achieve a factorization modulo smoothing operators
using two self-adjoint pseudodifferential operators. The following theorem shows that this is possible only in very special cases.

\begin{lem}
 Suppose there exist self-adjoint classical pseudodifferential operators $Q_1,Q_2$ on $\Sigma$ such that
 $$
  ( P - 2 \I \lambda X - \lambda^2) + (Q_1 -\lambda)(Q_2 -\lambda) = K(\lambda),
 $$
 where $K$ is a polynomial family of smoothing operators.
 Then $\mathcal{L}_\beta \tilde h =0$, i.e. the shift vector field is a Killing field for the metric $\tilde h$.
\end{lem}
\begin{proof}
 Considering the equation on the level of principal symbols one obtains that $Q_1, Q_2$ must be first order operators with
 \begin{gather*}
  \sigma_{Q_1} \sigma_{Q_1}  = -\sigma_P = -\tilde h, \quad \sigma_{Q_1} + \sigma_{Q_2} = - 2\rmi \sigma_X=2 \beta. 
 \end{gather*}
 This has the solution (unique up to interchanging the symbols of $Q_1$ and $Q_2$)
 $$
  \sigma_{Q_\pm} =  \beta \pm \sqrt{ h + \beta^2}.
 $$
 Now note that the pencil is self-adjoint in the sense that  $( P - 2 \I \lambda X - \lambda^2)^*= P - 2 \I \overline \lambda X - \overline\lambda^2$.
 This gives
 $$
  (Q_2 -\lambda)(Q_1 -\lambda) =  (Q_1 -\lambda)(Q_2 -\lambda) \mod C^\infty
 $$
 and therefore
 $$
  Q_2 Q_1 = Q_1 Q_2  \mod C^\infty.
 $$
 The commutator $[Q_1,Q_2]$ equals $[Q_1+Q_2,Q_2]$ and its principal symbol is therefore equal to the Poisson bracket $-\rmi \{2 \beta,  \beta - \sqrt{ h + \beta^2}\}$.
 Therefore we obtain
 $$
  \mathcal{L}_\beta \left( \beta - \sqrt{ h + \beta^2} \right) =0
 $$
 which is equivalent to $\mathcal{L}_\beta h =0$.
 \end{proof}

This means that the eigenvalues of $D_Z$ do not in general reduce in a simple way to eigenvalues of a self-adjoint scalar operator. For the purposes of a spectral asymptotics it might be possible to treat the operator pencil
using classical pseudodifferential methods, for example by analyzing the pencil as a parameter dependent pseudodifferential operator, such an approach would be non-covariant and somewhat unnatural.

\section{Examples} \label{examples}

The simplest examples are product space-times in which case our formula reduces to the Duistermaat-Guillemin trace formula.

\subsection{Ultrastatic Spacetimes} \label{DG-Product}

Assume that $M = \R \times \Sigma$ has metric $g = -dt^2 + h$, where $(\Sigma,h)$ is a connected $d$-dimensional closed Riemannian manifold. Then $(M,g)$ is a globally hyperbolic stationary spatially compact spacetime with Killing vector field $\partial_t$. The dimension $n$ of $M$ is $n=d+1$.

Separation of variables shows that the eigenvectors of eigenvalue $\pm \lambda_j $ of $-\rmi \partial_t$  on $\ker \Box$ are the functions of the form $\psi_j(t,x) = \ee^{\pm \rmi \lambda_j t} \phi_j$, where $\phi_j$ is an eigenfunctions of the Laplace operator $-\Delta$ on $\Sigma$ with eigenvalue $\lambda_j^2$. The function $t$ is also in $\ker \Box$ and it is a generalised eigenvector of $D_Z$ with eigenvalue $0$.
The spectrum of $D_Z$ on $\ker \Box$ consists therefore of the eigenvalues $\pm \lambda_j$ and the subspace $\mathrm{span} \{1,t \}$ gives the only non-trivial Jordan-block in the spectral decomposition of $D_Z$.

In this case the retarded and advanced fundamental solutions can be written explicitly in terms of functions of the Laplace operator and convolution in the time variable.
We have
\begin{gather}
 E_\mathrm{ret} (t,t') = \Theta(t-t') \Delta^{-\frac{1}{2}} \sin{ (t-t') \Delta^{\frac{1}{2}}},\\
 E_\mathrm{adv} (t,t') = -\Theta(t'-t) \Delta^{-\frac{1}{2}} \sin{ (t-t') \Delta^{\frac{1}{2}}},\\
 E (t,t') = \Delta^{-\frac{1}{2}} \sin{ (t-t') \Delta^{\frac{1}{2}}},
\end{gather}
where $\Theta$ is the Heaviside function. A decomposition $E= S_+ - S_-$ would be given by
\begin{gather} \label{productform}
 S_- = -\rmi \left( \frac{1}{2} \Delta^{-\frac{1}{2}} \exp( - \rmi  (t-t') \Delta^{\frac{1}{2}}) (1-P_0) + \frac{1}{2}(1 - \rmi t)(1+\rmi t')  P_0\right), \\ 
 S_+ =  -\rmi \left( \frac{1}{2} \Delta^{-\frac{1}{2}} \exp( + \rmi  (t-t') \Delta^{\frac{1}{2}}) (1-P_0) + \frac{1}{2}(1 + \rmi t)(1-\rmi t') P_0\right),
\end{gather}
where $P_0$ is the orthogonal projection onto $\ker \Delta$. There is no decomposition that is invariant under the Killing flow because of the non-trivial Jordan block.

The half-wave group on $(\Sigma,h)$ is the unitary group $V(t) = \ee^{\rmi t \sqrt{-\Delta}}$. The trace of the half-wave group of a compact Riemannian manifold is the
distribution trace,
$$\Tr V(t) = \sum_{ \lambda_j \in Sp(\sqrt{-\Delta})} \ee^{\rmi t \lambda_j}.$$ By the above we have $\Tr U(t) = \Tr V(t) + \Tr V(-t)$.
The singular points of $\Tr V(t)$ occur when $t$ lies in the length spectrum $\mathrm{Lsp}(\Sigma,h)$,
i.e. the set of lengths of closed geodesics. We denote the length of a closed geodesic
$\gamma$ by $L_{\gamma}.$  For each $L = L_{\gamma} \in \mathrm{Lsp}(\Sigma,h)$ there are at
least two closed geodesics of that length, namely $\gamma$ and $\gamma^{-1}$ (its
time reversal).  The singularities due to these lengths are identical so one often
considers the even part of $\Tr V(t)$ i.e. $\Tr E(t) =\frac{1}{2} \Tr U(t)$ where $E(t)= \cos (t \sqrt{-\Delta})$.

A Riemannian manifold is said to be `bumpy' if closed geodesics if  all closed geodesics are isolated and non-degenerate.
The trace of the wave group on a compact, bumpy Riemannian manifold $(\Sigma,h)$ has the
singularity expansion
\begin{equation} \Tr V(t) = e_0(t) + \sum_{L \in \mathrm{Lsp}(\Sigma,h)} e_L(t)
\end{equation}
where the sum runs over 
with
$$e_0(t) = a_{0,-d}(t+ \rmi 0)^{-d} + a_{0, -d+1}(t+ \rmi 0)^{-d+1}+\ldots $$
\begin{equation}  \begin{array}{lll}
e_L(t) &=& a_{L,-1} (t-L+ \rmi 0)^{-1} + a_{L,0}\log (t-(L+\rmi 0))\\[10pt]
&+&a_{L,1} (t-L+\rmi 0)\log (t-(L+\rmi 0)) +\ldots\;\;,\end{array} \end{equation}
where $\ldots$ refers to terms of ever higher degrees \cite{DG75}.
   The principal wave invariant at $t = L$ in the case of a
non-degenerate closed geodesic is given by \eqref{TrUt},
 $$a_{L,-1} = \sum_{\gamma:L_{\gamma}=L} \frac{1}{2 \pi \rmi}
\frac{\ee^{-\frac {\rmi \pi }{2} m_\gamma}L_\gamma^{\#}}{|\det(I-P_{\gamma})|^{\half}},$$
where $\{\gamma\}$ runs over the set of closed geodesics, and where $L_\gamma$,
$L_\gamma^{\#}$, $m_\gamma$, resp.\ $P_\gamma$ are the length, primitive length, Maslov index and linear
Poincar\'e map of $\gamma$.

\subsection{Static spacetimes}
Let $(\Sigma,h)$ be a $(n-1)$-dimensional closed Riemannian manifold and let $N \in C^\infty(\Sigma)$ be a positive smooth function. Then, the spacetime $\R \times \Sigma$ with metric
$$
  g=N^2 (-\dd t^2 + h)
$$
is stationary and globally hyperbolic. One computes
$$
N^{1+\frac{n}{2}} \Box_g N^{1-\frac{n}{2}} =  \partial_t^2 - \Delta_{h} + V, 
$$
where the potential $W$ is given by
$$
 V = N^{\frac{n}{2}-1} (\Delta_h N^{1-\frac{n}{2}}).
$$
Hence, the eigenvalues are the positive and negative square roots of the eigenvalues of the self-adjoint operator  $-\Delta_{h} + V$ on $\Sigma$.
This operator always has zero as an eigenvalue with eigenfunction $\phi_0=N^{1-\frac{n}{2}}$. As in the ultrastatic case this gives rise to the two dimensional generalised eigenspace
for $D_Z$ with eigenvalue zero that is spanned by the functions $1$ and $t$. Hence in this case there is a non-trivial Jordan block in the decomposition of $D_Z$. Since the energy quadratic form is non-negative
it follows that even though $V$ may fail to be non-negative the operator $-\Delta_{h} + V$ never has negative eigenvalues.

\subsection{Stationary pp-wave spacetimes}

Let $H: \R^{n-1} \to \R$ be a smooth function. Then the metric
$$
 - H(t,y) \dd t^2 + 2 \dd t \dd x + dy^2
$$
on $\R_t\times \R_x \times \R^{n-2}_y$ is called pp-wave metric in Brinkmann form. If $H$ is harmonic in $y$ this metric is a vacuum solution of the Einstein equations. 
It admits a lightlike Killing vectorfield $\partial_x$. If $H$ does not depend on $t$ then also $\partial_t$ is a Killing vectorfield. In case $H$ is positive this Killing vectorfield is timelike.

Here we will consider a modified situation. Namely, let  $(S,h)$ be a $(n-2)$-dimensional Riemannian manifold and let $H \in C^\infty(\R \times S)$ be a positive smooth function.
Let us consider the space $\R^2_{t,x} \times S_y$ with metric
$$
 g = - H(x,y) \dd t^2 + 2 \dd t \dd x + h.
$$

The orthogonal complement of the timelike Killing vectorfield $\partial_t$ is spanned by the vectors 
$\partial_t + H \partial_x, \partial_{y_1},\ldots,\partial_{y_{n-2}}$. The distribution defined by these vectors is not in general integrable
if $H$ depends on $y$. Hence, generically this spacetime is not static.
Note that $\partial_x$ is lightlike but not a Killing field unless $H$ is independent of $x$. If $S$ is flat and $H$ independent of $x$ then this is locally a stationary pp-wave spacetime.
We will restrict ourselves to the case when $(S,h)$ is closed and $H$ is periodic in $x$ with period $L$. Suppose that $\alpha \in \R$ is fixed such that $\alpha > \frac{1}{2} H$ everywhere.
In this case the vector $\partial_t + \alpha \partial_x$ is spacelike and the 
map $(t,x,y) \to (t+L,x+\alpha L,y)$ generates a group $\Gamma$ of isometries. We
can consider the quotient $M=(\R^2 \times S)/\Gamma$. This quotient is a spatially compact stationary globally hyperbolic spacetimes with Killing vectorfield $\partial_t$. 
For example, $(t,\alpha t + c,y)_{c \in \R}$ provides a foliation by Cauchy surfaces.

One computes the wave operator in local coordinates
$$
 \Box_g  = -2 \partial_t \partial_x - (\partial_x H(x,y)) \partial_x - H(x,y) \partial^2_x - \Delta_h.
$$
Separation of variables shows that $\lambda$ is an eigenvalue of $D_Z$ if there exists a non-zero function $\psi(x,y)$ on $\R \times S$ such that
\begin{gather*}
 (-2 \rmi \lambda \partial_x -  (\partial_x H) \partial_x - H \partial^2_x - \Delta_h) \psi =0,\\
 \psi(x+ \alpha L,y) = \ee^{-\I \lambda L} \psi(x,y).
\end{gather*}

There is family of eigenvalues $\lambda = \frac{2 \pi m}{L}$ indexed by $m \in \Z$ with corresponding eigenfunctions $e^{\frac{2 \pi \rmi m}{L} t}$
independent of $x$ and $y$. These eigenvalues give rise to a singularity in the wave trace at integer multiples of $L$. Note that 
the coordinate curve $(0,s,0,0)$ is a lightlike geodesic and the assignmet $(t,x,y) \mapsto (t+L,x,y)$ maps this geodesic to itself. Hence, $L$ is the length of a primitive
periodic orbit of $\mathcal{N}$.

It is instructive to see what happens in the case of a stationary pp-wave, i.e. when $H$ is independent of $x$ and $(S,h)$ is a flat torus
$\R^{n-2}/\Lambda$, where $\Lambda$ is a cocompact lattice in $\R^{n-2}$. Then
$\lambda$ is an eigenvalue of $D_Z$ if and only if there exists an integer $m \in \Z$ such that zero is an eigenvalue of the operator
$-\Delta_y + V(\lambda,y,m)$, where $$V(\lambda,y,m)=\frac{(2\pi m + L \lambda)^2}{\alpha ^2 L^2} H(y) - \frac{2 \lambda}{\alpha L}(2 \pi m + L \lambda).$$
The lightlike geodesics that are not of the form $(0,s,0)$ can be parametrised as $(s,x(s),y(s))$ and then satisfy the equations of motion
$$
 \left( \begin{matrix} \ddot x \\ \ddot y \end{matrix} \right) = \left( \begin{matrix}  \dot y \cdot H'(y) \ \\ -\frac{1}{2} H'(y) \end{matrix} \right).
$$
The unique solution with future directed lightlike initial data $(1,\dot x_0, \dot y_0)$ at the point $(t,x_0,y_0)$ is $$(t+s, x_0 + \int_0^s (H(y(u)) -E  ) d u, y(s)),$$ where 
$(y(s),\dot y(s))$ is the flow along the Hamiltonian vector field of $\frac{\xi^2}{2} + W$ starting at $(y_0,\dot y_0)$ with energy $E= \frac{\dot y_0^2}{2} + W(y_0)$, and $W$ is the potential
$W(y)=\frac{1}{2} H(y)$. Therefore, for a trajectory to be periodic of period $T$ we must have that $y(s)$ is periodic classical trajectory on $T^*S$ for the Hamiltonian $\frac{1}{2} p^2 + W$ of energy $E$ and period $\ell$ such 
that there exists an integer $k \in \Z$ with
\begin{gather*}
 T = -\ell + k L,\\
 \int_0^T (H(y(s)) -E) ds = \alpha k L. 
\end{gather*}
For example, in case $k=0$ this singles out periodic orbits of length $\ell$ for which
$$
 \int_0^\ell (H(y(s)) -E) ds = 0.
$$
These are orbits for which the energy is twice the average of the potential along the trajectory. This means the initial velocity
is the same as the average of the potential along the orbit. Unlike semi-classical analysis that singles out an energy shell this is a non-local condition.

\section{\label{APPENDIX} Appendix on trace formula and symbol composition }

\subsection{\label{SYMBOLAPP}Symbol composition}
Let $V, W$ be symplectic vector spaces and let $\Gamma$ be a Lagrangian
subspace of $V \times W$. Let $\Lambda$ be a Lagrangian subspace of
$W$. Let 
\begin{equation}
\Gamma \circ \Lambda = \{v \in V \colon \text{there exists $(v,w) \in \Gamma$ with $w \in  \Lambda$} \}.
\end{equation}
Let $\pi \colon \Gamma \to W$ and $\rho \colon \Gamma \to V$ be the coordinate
projections. 
Let  $F = \{(a = (v,w) ,b = w) \in \Gamma \times \Lambda, \pi(a) =w  =\iota (b) \in W\}$ be
the fiber product.
Let $\alpha $ be the composite map $$\alpha: F \to \Gamma \stackrel{\rho}{\rightarrow} V, \qquad \alpha(v,w, b) = \rho(v,w)=v. $$

\begin{prop} \label{LINALG_c07} 
$\Gamma \circ \Lambda$ is a symplectic subspace of $W$,
and  there is a canonical isomorphism,
\begin{equation} \label{HALFDCLAIM_c07} 
 |\Lambda|^{\half} \otimes |\Gamma|^{\half} \simeq |\ker
\alpha| \otimes |\Gamma \circ \Lambda|^{\half}.
 \end{equation}
\end{prop}

\begin{proof}
First we have the exact sequence
\begin{equation}\label{5.2}
0 \to \ker \alpha \to F \stackrel{\alpha}{\rightarrow} \Gamma \circ \Lambda \to 0,
\end{equation}
which implies
\begin{equation} \label{EX1_c07} 
|F|^{\half} \simeq |\Gamma \circ \Lambda|^{\half} \otimes |\ker \alpha|^{\half}.
\end{equation}

Define 
\begin{equation}
\tau \colon \Gamma \times \Lambda \to W, \qquad \tau( (v,w), b ) = \pi(v,w) - \iota (b) = w -b.
\end{equation}
 Then the following is an  exact sequence
\begin{equation}
0 \to F \to \Gamma \times \Lambda \stackrel{\tau}{\rightarrow} W \to \operatorname{coker} \tau \to 0, 
\end{equation} 
which implies
\begin{equation}
|F|^{-\half} \otimes |\Gamma|^{\half} \otimes |\Lambda|^{\half}
\otimes |W|^{-\half} \otimes |\operatorname{coker} \tau|^{\half} \simeq 1,
\end{equation}
hence
\begin{equation} \label{EX2_c07}
|F|^{\half} \otimes |W|^{\half} \otimes  |\operatorname{coker} \tau|^{-\half} 
 \simeq |\Gamma|^{\half} \otimes |\Lambda|^{\half}.
\end{equation}
 Combining \eqref{EX1_c07} and \eqref{EX2_c07} gives
 \begin{equation} \label{EX3_c07}
|\Gamma|^{\half} \otimes |\Lambda|^{\half}
 \simeq  |\Gamma \circ \Lambda|^{\half} \otimes |\ker \alpha|^{\half}  \otimes |W|^{\half} \otimes  |\operatorname{coker} \tau|^{-\half} .
\end{equation}

 To complete the proof we need to show that
 \begin{equation} \label{EX4_c07} 
|\operatorname{coker}\tau|^{-\half}  \simeq  |\ker \alpha|^{\half}.
\end{equation}
This follows from the fact that 
$\ker \alpha$ and $\operatorname{coker}\tau$ are dually paired by the symplectic
form on $W$, so that $(\ker \alpha)^{\perp} = \Im \tau$. Indeed, $\ker \alpha = \{(a = (v,w'),w) \in F \colon  \rho(a)  = v = 0\}$ and  $(a,w) \in F$ if and only if $w' = w$. Hence
$\ker \alpha \simeq  \{w \in  \Lambda \colon (0,w) \in \Gamma\}$.
On the other hand, if
$u \in \Im \tau$, then $u = w_2 - w_1$ with $(v_2, w_2) \in \Gamma$ and $w_1 \in \Lambda$. Now suppose that in the identification above $w \in \ker \alpha$,
and $u \in \Im \tau$. 
Then $ \Omega_W(w_1, w) = 0$ since $w_1, w \in \Lambda$ and $\Lambda$ is Lagrangian. Moreover,  $\Omega_W(w_2 ,w) = 0 $ since $\Gamma $ is Lagrangian in $V \times W$ and so $\{w \colon (0, w) \in \Gamma\}$ is isotropic
in $W$.  Hence,
$\Omega_W(w, u) = 0$. Since $\Gamma$ and $\Lambda$ are Lagrangian,
it follows that $(\ker \alpha)^{\perp} = \Im \tau$ in $W$. This implies
\eqref{EX4_c07}. Since $|W|^{\half} \simeq 1$ (i.e., there is a canonical choice of half-density),
 this proves \eqref{HALFDCLAIM_c07}.
\end{proof}

In the case where $\Gamma$ is the diagonal and $\Lambda $ is the graph of $P$, the exact
sequence \eqref{5.2} boils down to the exact sequence 
\begin{equation} \label{4.1} 0 \rightarrow \ker (I - \acal) \to  V \stackrel{I - \acal}{\rightarrow} V \to V/{\rm Im}(I - \acal) \to 0. \end{equation}

\end{document}